\documentclass[11pt,leqno]{amsart}
\usepackage{verbatim,amssymb,amsfonts,latexsym,amsmath,amsthm,tikz,bbm,nicefrac,xspace,graphicx,mathtools,mathdots,enumerate,multicol}
\usetikzlibrary{arrows,hobby}
\usetikzlibrary{decorations.pathmorphing}
\usetikzlibrary{decorations.pathreplacing}

\tikzset{snake it/.style={decorate, decoration=snake}}

\newtheorem{theorem}{Theorem}[section]
\newtheorem{lemma}[theorem]{Lemma}
\newtheorem{corollary}[theorem]{Corollary}

\newtheorem*{maintheorem1}{Main Theorem}
\newtheorem*{maintheorem2}{Theorem}
\newtheorem*{theorem*}{Theorem}
\newtheorem*{corollary*}{Corollary}
\newtheorem*{claim}{Claim}
\newtheorem*{sub-claim}{sub-claim}
\theoremstyle{definition}
\newtheorem{example}[theorem]{Example}
\newtheorem{definition}[theorem]{Definition}
\newtheorem*{notation}{Notation}

\theoremstyle{remark}

\newcommand{\N}{\mathbb{N}}

\newcommand{\A}{\mathcal{A}}
\newcommand{\B}{\mathcal{B}}
\newcommand{\C}{\mathcal{C}}

\newcommand{\F}{\mathcal{F}}
\newcommand{\G}{\mathcal{G}}
\newcommand{\HH}{\mathcal{H}}
\newcommand{\I}{\mathcal{I}}
\newcommand{\J}{\mathcal{J}}

\renewcommand{\P}{\mathcal{P}}

\newcommand{\explicitSet}[1]{\left\lbrace #1 \right\rbrace}
\newcommand{\brackets}[1]{\left\langle #1 \right\rangle}
\newcommand{\set}[2]{\explicitSet{#1 \colon #2}}
\newcommand{\seq}[2]{\brackets{#1 \colon #2}}

\newcommand{\g}{\gamma}

\newcommand{\s}{\sigma}

\newcommand{\w}{\omega}
\newcommand{\0}{\emptyset}
\newcommand{\sub}{\subseteq}
\newcommand{\rest}{\!\restriction\!}

\newcommand{\closure}[1]{\overline{#1}}

\newcommand{\card}[1]{\left\lvert #1 \right\rvert}

\newcommand{\continuum}{\mathfrak{c}}

\newcommand{\ser}{\ss}
\newcommand{\Ser}{I}
\newcommand{\Hyp}{H}

\makeatletter
\newcommand{\specialcell}[1]{\ifmeasuring@#1\else\omit$\displaystyle#1$\ignorespaces\fi}
\makeatother

\begin{document}

\title{Choosing between incompatible ideals}
\author[W. R. Brian]{Will Brian}
\address {
W. R. Brian\\
Department of Mathematics and Statistics\\
University of North Carolina at Charlotte\\
9201 University City Blvd.\\
Charlotte, NC}
\email{wbrian.math@gmail.com}
\urladdr{wrbrian.wordpress.com}
\author[P. B. Larson]{Paul B. Larson}
\address[P.~B.~Larson]{Department of Mathematics, Miami University,
  Oxford, OH 45056, U.S.A.}
\email{larsonpb@miamioh.edu}
\urladdr{http://www.users.miamioh.edu/larsonpb/}
\thanks{We wish to thank Louis DeBiasio for his input on an earlier draft of this paper, especially his thoughts concerninbg the numbers $\Hyp(n)$.}

\begin{abstract}
Suppose $\I$ and $\J$ are proper ideals on some set $X$. We say that $\I$ and $\J$ are \emph{incompatible} if $\I \cup \J$ does not generate a proper ideal. Equivalently, $\I$ and $\J$ are incompatible if there is some $A \sub X$ such that $A \in \I$ and $X \setminus A \in \J$. If some $B \subseteq X$ is either in $\I \setminus \J$ or in $\J \setminus \I$, then we say that $B$ \emph{chooses} between $\I$ and $\J$.

We consider the following Ramsey-theoretic problem: Given several pairs $(\I_1,\J_1), (\I_2,\J_2), \dots, (\I_k,\J_k)$ of incompatible ideals on a set $X$, find some $A \sub X$ that chooses between as many of these pairs of ideals as possible. 
The main theorem is that for every $n \in \N$, there is some $\Ser(n) \in \N$ such that given at least $\Ser(n)$ pairs of incompatible ideals on any set $X$, there is some $A \sub X$ choosing between at least $n$ of them.

This theorem is proved in two main steps. The first step is to identify a (purely finitary) problem in extremal combinatorics, and to show that our problem concerning ideals is equivalent to this combinatorial problem. The second step is to analyze the combinatorial problem in order to show that the number $\Ser(n)$ described above exists, and to put bounds on it. We show 
$$\textstyle \frac{1}{2}n \log_2 n - O(n) \,<\, \Ser(n) \,<\, n \ln n + O(n).$$

The upper bound is proved by considering a different but closely related combinatorial problem involving hypergraphs, which may be of independent interest.
We also investigate some applications of this theorem to a problem concerning conditionally convergent series.
\end{abstract}

\maketitle


\section{Introduction}

An \emph{ideal} on a set $X$ is a set $\I$ of subsets of $X$ such that
\begin{itemize}
\item[$\circ$] every finite subset of $X$ is in $\I$,
\item[$\circ$] if $A,B \in \I$ then $A \cup B \in \I$, and
\item[$\circ$] if $A \in \I$ and $B \sub A$, then $B \in \I$.
\end{itemize}
An ideal $\I$ is \emph{proper} if $\I \neq \P(X)$ or, equivalently, if $X \notin \I$. 
Two proper ideals $\I$ and $\J$ on a set $X$ are \emph{incompatible} if their union is not contained in any proper ideal. Equivalently, $\I$ and $\J$ are incompatible if there is some $A \sub X$ such that $A \in \I$ and $X \setminus A \in \J$. If $B \sub X$ and if either $B \in \I \setminus \J$ or $B \in \J \setminus \I$, then we say that $B$ \emph{chooses} between $\I$ and $\J$.

This paper investigates the following Ramsey-theoretic problem: given several pairs $(\I_1,\J_1), (\I_2,\J_2), \dots, (\I_k,\J_k)$ of incompatible ideals on a set $X$, find some $A \sub X$ that chooses between as many of these pairs of ideals as possible. In other words, we would like to make many choices simultaneously, using a single $A \sub X$.

\begin{maintheorem1}
For each $n \in \N$, there is some $\Ser(n) \in \N$ such that for any collection of $\Ser(n)$ pairs of incompatible ideals on a set $X$, there is some $A \sub X$ that chooses between at least $n$ of those pairs.
\end{maintheorem1}

This theorem is proved in two stages. The first stage is to identify a (purely finitary) problem in extremal combinatorics, and to show that our problem concerning ideals is equivalent to this combinatorial problem. The second stage is to analyze the combinatorial problem to show that the number $\Ser(n)$ exists, and to put bounds on it. Namely, we show that 
$$\textstyle \frac{1}{2} n \log_2 n - \frac{1}{2}n + \frac{3}{2} \,<\, \Ser(n+1) \,<\, 1+\sum_{k=1}^n\frac{n}{k}$$
Let us note that $\sum_{k = 1}^n \frac{n}{k} \,<\, n \ln n + \g n + \frac{1}{2}$,
where $\g \approx .5772156649$ is the Euler-Mascheroni constant, so that the upper and lower bounds on $I(n)$ match up to a constant factor. In other words, $\Ser(n) = \Theta(n \log n)$.

The first stage of this proof is contained in Section~\ref{sec:reduction}, and the second in the first part of Section~\ref{sec:combinatorics1}. The second part of Section~\ref{sec:combinatorics1} and all of Section~\ref{sec:combinatorics2} are devoted to finding upper and lower bounds for $\Ser(n)$. The upper bound is proved by introducing a second problem of extremal combinatorics concerning hypergraphs. Roughly, the problem is, given some hypergraph $(V,\HH)$, to find some $X \sub V$ and $\HH' \sub \HH$, with $X$ as large as possible, so that $\HH'$ induces a partition on $X$. This problem, which may have some independent interest, is described more precisely and analyzed in Section~\ref{sec:combinatorics1}.

In Section~\ref{sec:series}, we apply the main theorem to a problem concerning infinite series. 
A \emph{conditionally convergent series} is a sequence $\seq{a_n}{n \in \N}$ of real numbers such that $\sum_{n \in \N}a_n$ converges but $\sum_{n \in \N}|a_n|$ does not. Every conditionally convergent series $\seq{a_n}{n \in \N}$ has a subseries $\seq{a_{n_k}}{k \in \N}$ summing to $\infty$ (e.g., the subseries obtained by summing only over positive terms), and a different subseries summing to $-\infty$ (e.g., the subseries obtained by summing only over negative terms). Every conditionally convergent series also has many subseries that diverge by oscillation; for example, one may construct such a subseries by interleaving long stretches of positive terms with long stretches of negative terms.

\begin{definition}
Let us say that a set $A \sub \N$ \emph{sends a series $\sum_{n \in \N}a_n$ to infinity} if either $\sum_{n \in A}a_n = \infty$ or $\sum_{n \in A}a_n = -\infty$. 
\end{definition}

\noindent Notice that in this definition, we require each subseries to diverge either to $\infty$ or to $-\infty$, and not merely to diverge by oscillation.
Recently, the first author has investigated the problem of sending several series to infinity simultaneously, using a single $A \sub \N$. In summary: 
\begin{itemize}
\item[$\circ$] For any three conditionally convergent series, there is a single $A \sub \N$ sending all three series to infinity simultaneously. (Brian \cite{Brian})
\item[$\circ$] The analogous statement for four series is false: there is a collection of four conditionally convergent series such that no single $A \sub \N$ sends all four series to infinity. (Nazarov \cite{Nazarov}, \cite[Section 3]{Brian})
\end{itemize}

These results suggest the following Ramsey-theoretic problem: given several conditionally convergent series, find some $A \sub X$ that sends as many of them as possible to infinity. In other words, we would like to send many series to infinity simultaneously, using a single $A \sub X$.
As a corollary to the main theorem we obtain:

\begin{maintheorem2}
For each $n \in \N$, there is some $\ser(n) \in \N$ such that for any collection of $\ser(n)$ conditionally convergent series, there is some $A \sub \N$ sending at least $n$ of the series to infinity simultaneously.
\end{maintheorem2}

Specifically, we show that $\ser(n) \leq \Ser(n) = \Theta(n \log n)$. The best lower bound for $\ser(n)$ that we know is $2n-5 \leq \ser(n)$, which is proved in Section~\ref{sec:series} below. This is accomplished by generalizing the aforementioned example of Navarov to find, for every $n \geq 2$, an example of $2n$ conditionally convergent series such that no single $A \sub \N$ can send more than $n+2$ of them to infinity simultaneously.
Finally, Section~\ref{sec:infinite} contains an infinitary version of the result on conditionally convergent series:

\begin{theorem*}
For any infinite collection of conditionally convergent series, there is a single $A \sub \N$ sending infinitely many of the series to infinity.
\end{theorem*}

\section{The reduction}\label{sec:reduction}

In this section we reduce the main theorem stated in the introduction to a (purely finitary) problem of extremal combinatorics. One may think of this as the first half of a proof of the main theorem. The second half, contained in the next section of the paper, is to analyze the combinatorial problem. 

We begin with a formal definition of $\Ser(n)$:

\begin{definition}\label{def:full}
For each $n \in \N$, let $\Ser(n)$ denote the least $k \in \N$ with the following property:
\begin{itemize}
\item[$(*)_n$] For any set $X$ and any collection $(\I_1,\J_1),(\I_2,\J_2),\dots,(\I_k,\J_k)$ of $k$ pairs of incompatible ideals on $X$, there is some $A \sub X$ such that $A$ chooses between $\I_\ell$ and $\J_\ell$ for at least $n$ different values of $\ell$.
\end{itemize}
If there is no such $k \in \N$, then we say that $\Ser(n)$ is not well-defined.
\end{definition}

We now define another number $\tilde \Ser(n)$ for each $n \in \N$. As with $\Ser(n)$, we define it in such a way that it is not yet assumed to be well-defined. The notation $\tilde \Ser(n)$ is used only in this section; it is not needed in later sections because, as we shall show shortly, $\tilde \Ser(n)$ and $\Ser(n)$ are the same.

\begin{definition}\label{def:full}
Let $\mathrm{Fn}(k,2)$ denote the set of all functions from a subset of $\{1,2,\dots,k\}$ to $\{p,n\}$. We say that $\F \sub \mathrm{Fn}(k,2)$ is \emph{full} provided that, for every $i \in \{1,2,\dots,k\}$ and $j \in \{p,n\}$, there is some $f \in \F$ with $f(i) = j$. For each $n \in \N$, let $\tilde \Ser(n)$ denote the least $k \in \N$ with the following property:
\begin{itemize}
\item[$(\dagger)_n$] For any full $\F \sub \mathrm{Fn}(k,2)$, there is some $\G \sub \F$ and $D \sub \bigcup_{g \in \G}\mathrm{dom}(g)$ such that $|D| \geq n$, and no two members of $\G$ disagree on $D$.
\end{itemize}
If there is no such $k \in \N$, then we say that $\tilde \Ser(n)$ is not well-defined.
\end{definition}

In what follows, we shall sometimes represent functions as sets of ordered pairs, and sometimes we shall represent functions and sets of functions as pictures. For example the following two pictures represent subsets of $\mathrm{Fn}(3,2)$:

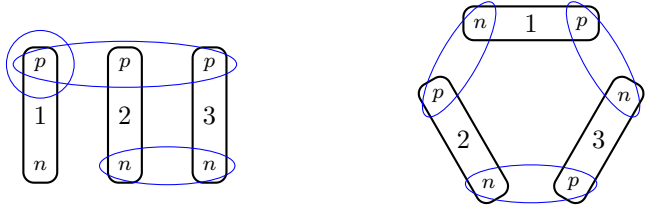
\vspace{-1mm}
\begin{center}
\begin{tikzpicture}[xscale=.9,yscale=.9]

\begin{scope}[shift={(-7.5,1.1)}]

\draw[rounded corners,thick] (0,0) rectangle (.5,2) {};
\draw[rounded corners,thick] (1.25,0) rectangle (1.75,2) {};
\draw[rounded corners,thick] (2.5,0) rectangle (3,2) {};
\node at (.25,.25) {\scriptsize $n$};
\node at (.25,1.75) {\scriptsize $p$};
\node at (1.5,.25) {\scriptsize $n$};
\node at (1.5,1.75) {\scriptsize $p$};
\node at (2.75,.25) {\scriptsize $n$};
\node at (2.75,1.75) {\scriptsize $p$};
\node at (.25,1) {\small $1$};
\node at (1.5,1) {\small $2$};
\node at (2.75,1) {\small $3$};

\draw[blue] (.25,1.75) circle (.5cm);
\draw[blue] (1.5,1.75) ellipse (1.65cm and .34cm);
\draw[blue] (2.125,.25) ellipse (1cm and .28cm);

\end{scope}

\draw[rounded corners,thick] (-1,3.206) rectangle (1,3.706) {};
\node at (-.75,3.45) {\scriptsize $n$};
\node at (.75,3.45) {\scriptsize $p$};
\draw[rotate=30,rounded corners,thick] (-.25,1) rectangle (.25,3) {};
\node at (-1.375,2.382) {\scriptsize $p$};
\node at (1.375,2.382) {\scriptsize $n$};
\node at (-.625,1.083) {\scriptsize $n$};
\node at (.625,1.083) {\scriptsize $p$};
\draw[rotate=-30,rounded corners,thick] (-.25,1) rectangle (.25,3) {};
\node at (0,3.456) {\small $1$};
\node at (-1,1.733) {\small $2$};
\node at (1,1.733) {\small $3$};

\draw[blue] (0,1.083) ellipse (.97cm and .28cm);
\draw[rotate=-60,blue] (-2,2.382) ellipse (.97cm and .28cm);
\draw[rotate=60,blue] (2,2.382) ellipse (.97cm and .28cm);

\end{tikzpicture}
\end{center}

\noindent The picture on the left represents the subset of $\mathrm{Fn}(3,2)$ containing the three functions $\{(1,p)\}$, $\{(1,p),(2,p),(3,p)\}$, and $\{(2,n),(3,n)\}$. The picture on the right shows a full subset of $\mathrm{Fn}(3,2)$. (The collection of functions represented on the left is not full.) The reader may check that the family represented by the picture on the right fails to have property $(\dagger)_3$. Thus this picture shows $\tilde \Ser(3) > 3$, assuming that $\tilde \Ser(3)$ is well-defined.


\begin{theorem}
For each $n \in \N$, the properties $(*)_n$ and $(\dagger)_n$ are equivalent. Hence $\Ser(n)$ is well-defined if and only if $\tilde \Ser(n)$ is, and if they are both well-defined then $\Ser(n) = \tilde \Ser(n)$.
\end{theorem}
\begin{proof}
We prove first that $(\dagger)_n$ implies $(*)_n$. Fix $k \in \N$ and suppose that $(\dagger)_n$ holds. Let $X$ be a set and let $(\I_1,\J_1),(\I_2,\J_2),\dots,(\I_k,\J_k)$ be pairs of incompatible ideals on $X$. 

For each $A \sub X$, define a function $\phi_A \in \mathrm{Fn}(k,2)$ as follows. For every $\ell \in \{1,2,\dots,k\}$, let $\ell \in \mathrm{dom}(\phi_A)$ if and only if $A$ chooses between $\I_\ell$ and $\J_\ell$, and for all $\ell \in \mathrm{dom}(\phi_A)$, let
$$\phi_A(\ell) = 
\begin{cases}
p & \text{ if } A \in \I_\ell \setminus \J_\ell \\
n & \text{ if } A \in \J_\ell \setminus \I_\ell.
\end{cases}$$

Let us say that $A \sub X$ is \emph{tame} if for all $\ell \in \{1,2,\dots,k\}$, either $A \in \I_\ell$ or $A \in J_\ell$.
Then let 
$$\F = \set{\phi_A}{A \sub X \text{ is tame}}.$$
Clearly, $\F \sub \mathrm{Fn}(k,2)$. 

\begin{claim}
$\F$ is full.
\end{claim}
\begin{proof}[Proof of claim]
For every $m \in \{1,2,\dots,k\} \setminus \{\ell\}$, use the fact that $\I_m$ and $\J_m$ are incompatible to fix some $A_m \sub X$ such that $A_m \in \I_m$ and $\N \setminus A_m \in \J_m$. For convenience, let $A^p_m$ denote $A_m$, and let $A^n_m$ denote $\N \setminus A_m$.

Fix $\ell \in \{1,2,\dots,k\}$, and partition $A_\ell$ into $2^{k-1}$ sets as follows. For each function $f$ from $\{1,2,\dots,k\} \setminus \{\ell\}$ to $\{p,n\}$, define
$$A_\ell^f \,=\, \textstyle \bigcap \set{A^{f(m)}_m}{m \in \{1,2,\dots,k\} \setminus \{\ell\}} \cap A_\ell.$$
There are only finitely many functions from $\{1,2,\dots,k\} \setminus \{\ell\}$ to $\{p,n\}$, so (because $\J_\ell$ is closed under finite unions) if every $A_\ell^f$ were in $\J_\ell$ then $A_\ell = \bigcup_{f: \{1,2,\dots,k\} \setminus \{\ell\} \to \{p,n\}} A^f_\ell$ would also be in $\J_\ell$. But $A_\ell \notin \J_\ell$, because $\N \setminus A_\ell \in \J_\ell$. Thus there is some $f$ such that $A^f_\ell \notin \J_\ell$.
On the other hand, $A_\ell^f \in \I_\ell$ because $A^f_\ell \sub A_\ell \in \I_\ell$. Hence $A^f_\ell$ chooses between $\I_\ell$ and $\J_\ell$, and in fact $\phi_{A_\ell^f}(\ell) = p$.

For each $m \in \{1,2,\dots,k\} \setminus \{\ell\}$, either $A^f_\ell \sub A^p_m$, in which case $A^f_\ell \in \I_m$, or $A^f_\ell \sub A^n_m$, in which case $A^f_\ell \in \J_m$. From this and the fact that $A^f_\ell \in \I_\ell$, it follows that $A^f_\ell$ is tame. 

Thus we have found a tame $A \sub X$ with $\phi_{A}(\ell) = p$.
A similar argument (which proceeds by partitioning $\N \setminus A_\ell$ instead of $A_\ell$) shows that there is a tame $A \sub X$ with $\phi_{A}(\ell) = n$.
As $\ell \in \{1,2,\dots,k\}$ was arbitrary, we conclude that $\F$ is full. 
\end{proof}

Applying $(\dagger)_n$, there is some $\G \sub \F$ and some $D \sub \bigcup_{g \in \G}\mathrm{dom}(g)$ such that $|D| \geq n$, and no two members of $\G$ disagree on $D$. For each $g \in \G$, fix some particular tame subset $A_g$ of $X$ such that $\phi_{A_g} = g$. Let $A = \bigcup_{g \in \G}A_g$.
We claim that $A$ chooses between the ideals $\I_\ell$ and $\J_\ell$ for every $\ell \in D$. 
Fix $\ell \in D$, and suppose that $g(\ell) = p$ for every $g \in \G$ with $\ell \in \mathrm{dom}(g)$. Then $A_g \in \I_\ell \setminus \J_\ell$ for all $g \in \G$ with $\ell \in \mathrm{dom}(g)$. 
For $\ell \notin \mathrm{dom}(g)$, $A_g \in \I_\ell$ and $A_g \in \J_\ell$ (because $A_g$ is tame).
Hence $A_g \in \I_\ell$ for all $g \in \G$.
As $\I_\ell$ is closed under finite unions, this implies $A \in \I_\ell$. 
Because $\J_\ell$ is closed under taking subsets, if $A \in \J_\ell$ then $A_g \in \J_\ell$ for every $g \in \G$. But this is not the case, because $\ell \in \mathrm{dom}(g)$ for at least one $g \in \G$, in which case $g(\ell) = p$ and $A_g \notin \J_\ell$. Thus $A \in \I_\ell \setminus \J_\ell$. Similarly, if $g(\ell) = n$ for every $g \in \G$ with $\ell \in \mathrm{dom}(g)$, then $A \in \J_\ell \setminus \I_\ell$. Either way, $A$ chooses between $\I_\ell$ and $\J_\ell$. Because this is true for every $\ell \in D$, and because $|D| \geq n$, the set $A$ chooses between at least $n$ of the pairs $(\I_1,\J_1),(\I_2,\J_2),\dots,(\I_k,\J_k)$. This completes the proof that $(\dagger)_n$ implies $(*)_n$.

\vspace{3mm}

We now prove the converse direction, that $(*)_n$ implies $(\dagger)_n$. Fix $k \in \N$, and suppose that $(*)_n$ holds. Let $\F \sub \mathrm{Fn}(k,2)$ be full. 

Let $X = \F \times \N$
and for each $\ell \leq k$ define two ideals $\I_\ell$ and $\J_\ell$ as follows:
$$\I_\ell = \set{A \sub X}{A \cap (\{f\} \times \N) \text{ is finite for every } f \in \F \text{ with } f(\ell) = p},$$
$$\J_\ell = \set{A \sub X}{A \cap (\{f\} \times \N) \text{ is finite for every } f \in \F \text{ with } f(\ell) = n}.$$

Fix $\ell \leq k$. Because $\F$ is full, there is some $f \in \F$ with $f(\ell) = p$. Therefore $X \notin \I_\ell$. All the other parts of the definition of an ideal are easy to check for $\I_\ell$, so $\I_\ell$ is an ideal on $X$. Similarly, the fullness of $\F$ implies that $\J_\ell$ is an ideal on $X$. To see that $\I_\ell$ and $\J_\ell$ are incompatible, consider
$$A = \textstyle \bigcup \set{\{f\} \times \N}{f \in \F \text{ and }f(\ell) \neq p}.$$
It is clear that $A \in \I_\ell$, and that
\begin{align*}
\N \setminus A & \,=\, \textstyle \bigcup \set{\{f\} \times \N}{f \in \F \text{ and }f(\ell) = p} \\
& \,\sub\, \textstyle \bigcup \set{\{f\} \times \N}{f \in \F \text{ and }f(\ell) \neq n} \in \J_\ell.
\end{align*}
Thus $\I_\ell$ and $\J_\ell$ are incompatible for each $\ell \leq k$.

Applying $(*)_n$, there is some $A \sub X$ such that $A$ chooses between at least $n$ of the pairs $(\I_1,\J_1),(\I_2,\J_2),\dots,(\I_k,\J_k)$. Let
$$\G = \set{f \in \F}{A \cap (\{f\} \times \N) \text{ is infinite}},$$
$$D = \set{\ell \leq k}{A \text{ chooses between } \I_\ell \text{ and } \J_\ell}.$$
By our choice of $A$, we have $|D| \geq n$. Thus, to prove $(*)_n$, it suffices to show that $D \sub \bigcup_{g \in \G}\mathrm{dom}(g)$ and that  no two members of $\G$ disagree on $D$. 

If $\ell \in D$, then $A$ chooses between $\I_\ell$ and $\J_\ell$, and in particular, either $A \notin \I_\ell$ or $A \notin \J_\ell$. Either way, this implies $A \cap (\{f\} \times \N)$ is infinite for some $f \in \F$ with $\ell \in \mathrm{dom}(f)$. But then $f \in \G$, so this shows $\ell \in \bigcup_{g \in \G}\mathrm{dom}(g)$. As $\ell$ was an arbitrary member of $D$, we have $D \sub \bigcup_{g \in \G}\mathrm{dom}(g)$.

If $\ell \in D$, then $A$ chooses between $\I_\ell$ and $\J_\ell$, and in particular, either $A \in \I_\ell$ or $A \in \J_\ell$. Suppose for now that $A \in \I_\ell$. 
The definition of $\I_\ell$ implies that $A \cap (\{f\} \times \N)$ is finite for every $f \in \F$ with $f(\ell) = p$; but then the definition of $\G$ implies that $f \notin \G$. Hence if $A \in \I_\ell$, then $f(\ell) \neq p$ for all $g \in \G$ (by which we mean that if $g \in \G$ then either $g(\ell) = n$ or $\ell \notin \mathrm{dom}(g)$). Similarly, if $A \in \J_\ell$, then $g(\ell) \neq n$ for all $g \in \G$. Either way, no two members of $\G$ disagree at $\ell$. As $\ell$ was an arbitrary member of $D$, no two members of $\G$ disagree on $D$.
\end{proof}

\section{Partitions in hypergraphs: an upper bound for $\Ser(n)$}\label{sec:combinatorics1}

In this section we prove the main theorem by showing that $\Ser(n)$ is well-defined for all $n \in \N$, and obtain the upper bound
$\Ser(n+1) \leq \textstyle 1+\sum_{k = 1}^n \frac{n}{k}$, which implies
$$\Ser(n) \,<\, \textstyle n \ln n + \g n - \ln n +\frac{3}{2} -\g,$$
where $\g \approx .5772156649$ is the Euler-Mascheroni constant.

To show that $\Ser(n)$ is well-defined and prove this upper bound, we do not analyze $\Ser(n)$ directly. Instead we introduce another problem of extremal combinatorics, defining another function $\Hyp(n)$. This function has a simpler definition than $\Ser(n)$, and seems significantly easier to work with. We show that $\Ser(n+1) \leq \Hyp(n)+1$ for all $n$, so that any upper bounds one might find for $\Hyp(n)$ automatically give upper bounds for $\Ser(n)$ also.

Recall that a \emph{hypergraph} is a set $V$ (called \emph{vertices}) together with a collection $\HH$ of subsets of $V$ (called \emph{hyperedges}). A hypergraph $(V,\HH)$ is said to \emph{contain a partition of size $n$} if there is some $D \sub V$ and $\P \sub \HH$ such that $|D| = n$ and every member of $D$ is contained in exactly one member of $\P$. In this case, we say that $D$ and $\P$ form a partition in $(V,\HH)$.

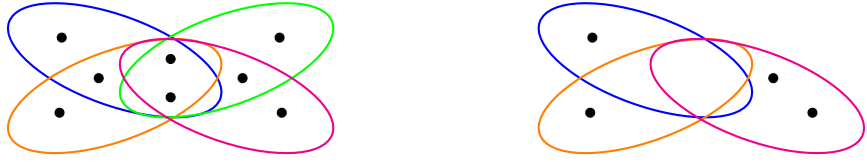
\vspace{1mm}
\begin{center}
\begin{tikzpicture}[xscale=.85,yscale=.85]

\draw[fill=black] (2,-.28) circle (2pt);
\draw[fill=black] (.875,.02) circle (2pt);
\draw[fill=black] (.875,-.58) circle (2pt);

\begin{scope}[rotate=-20]
\draw[thick,blue] (0,0) ellipse (50pt and 20pt);
\draw[fill=black] (-.9,.05) circle (2pt);
\end{scope}
\begin{scope}[shift={(0,-.56)}]\begin{scope}[rotate=20]
\draw[thick,orange] (0,0) ellipse (50pt and 20pt);
\draw[fill=black] (-.9,.05) circle (2pt);
\end{scope}\end{scope}

\begin{scope}[xscale=-1,shift={(-1.75,0)}]
\draw[fill=black] (2,-.28) circle (2pt);
\begin{scope}[rotate=-20]
\draw[thick,green] (0,0) ellipse (50pt and 20pt);
\draw[fill=black] (-.9,.05) circle (2pt);
\end{scope}
\begin{scope}[shift={(0,-.56)}]\begin{scope}[rotate=20]
\draw[thick,magenta] (0,0) ellipse (50pt and 20pt);
\draw[fill=black] (-.9,.05) circle (2pt);
\end{scope}\end{scope}\end{scope}

\begin{scope}[shift={(8.3,0)}]

\draw[fill=black] (2,-.28) circle (2pt);

\begin{scope}[rotate=-20]
\draw[thick,blue] (0,0) ellipse (50pt and 20pt);
\draw[fill=black] (-.9,.05) circle (2pt);
\end{scope}
\begin{scope}[shift={(0,-.56)}]\begin{scope}[rotate=20]
\draw[thick,orange] (0,0) ellipse (50pt and 20pt);
\draw[fill=black] (-.9,.05) circle (2pt);
\end{scope}\end{scope}

\begin{scope}[xscale=-1,shift={(-1.75,0)}]
\begin{scope}[shift={(0,-.56)}]\begin{scope}[rotate=20]
\draw[thick,magenta] (0,0) ellipse (50pt and 20pt);
\draw[fill=black] (-.9,.05) circle (2pt);
\end{scope}\end{scope}\end{scope}

\end{scope}


\end{tikzpicture}
\end{center}
\vspace{1mm}

\noindent For example, the picture on the left above shows a hypergraph with $8$ vertices and $4$ hyperedges. The picture on the right shows that it contains a partition of size $4$. 

Given a hypergraph $(V,\HH)$, a vertex $v \in V$ is called \emph{isolated} if $v \notin \bigcup \HH$.

Given a hypergraph, we will be interested in the problem of finding in it a partition involving as many vertices as possible. For example, in the hypergraph above we found a partition of size $4$, and one may check that this is the largest possible: i.e., there is no partition of size $\geq\! 5$.

\begin{definition}\label{def:sunflower}
Define $\Hyp(n)$ to be the greatest $k \in \N$ such that there is a hypergraph $(V,\HH)$ with $|V| = k$ having no isolated vertices and containing no partitions of size greater than $n$. If there is no such $k$, then we say that $\Hyp(n)$ is not well-defined.
\end{definition}

\noindent For example, because the hypergraph pictured above does not contain any partitions of size $>\!4$, it shows that if $\Hyp(4)$ is well-defined, then $\Hyp(4) \geq 8$.

\begin{theorem}\label{thm:Ser}
Let $n \in \N$. If $\F \sub \mathrm{Fn}(k,2)$ is full and $k \geq \Hyp(n)$, then $\F$ satisfies property $(\dagger)_{n+1}$. Consequently, if $\Hyp(n)$ is well-defined, then so is $\Ser(n+1)$, and 
$\Ser(n+1) \,\leq\, \Hyp(n)+1$.
\end{theorem}
\begin{proof}
Let $\F \sub \mathrm{Fn}(k,2)$ be full, and suppose that $k \geq \Hyp(n)+1$. Define a hypergraph by setting $V = \{1,2,\dots,k\}$ and
$$\HH = \set{\mathrm{dom}(f)}{f \in \F}.$$
Because $\F$ is full, $\bigcup \HH = \{1,2,\dots,k\}$. In other words, the hypergraph $(V,\HH)$ has no isolated points. Because $k > \Hyp(n)$, $(V,\HH)$ contains a partition of size greater than $n$. Fix some $\mathcal E \sub \HH$ and some $D \sub V$ with $|D| > n$ such that 
every member of $D$ is contained in exactly one member of $\mathcal E$.

For each $D \in \mathcal E$, choose a function $f_G \in \G$ with $\mathrm{dom}(f_G) = G$. Let
$\G = \set{f_G}{G \in \mathcal E}$. Then $D \sub \bigcup \mathcal E = \bigcup_{f \in \G}\mathrm{dom}(f)$, but no two members of $\G$ contain a common member of $D$. Hence all the functions in $\G$ agree on $D$, and as $|D| \geq n+1$, it follows that $\G$ satisfies property $(\dagger)_{n+1}$.
\end{proof}

To prove our paper's main theorem, it remains now to show that $\Hyp(n)$ is well-defined for every $n$. 
In fact, this is relatively easy: the more difficult part of this section is to prove an upper bound with leading term $n \log n$. 
The well-definedness of $\Hyp(n)$ can be proved in a paragraph or two by using the {sunflower lemma} of Erd\H{o}s and Rado \cite{Erdos&Rado}, or by using the {Sauer-Shelah lemma} \cite{Sauer,Shelah}. (Shelah attributes the lemma to Perles; it was proved independently, and possibly earlier, by Vapnik and \v{C}ervonenkis \cite{VC}). The proof we give presently is neither of these, however. While slightly longer, it has the advantage of being elementary and entirely self-contained. This proof of the well-definedness of $\Hyp(n)$ and $I(n)$ gives quadratic polynomials for their upper bounds.

If $(V,\HH)$ has no isolated points, but $(V,\HH \setminus \{E\})$ does have isolated points for any $E \in \HH$, then $(V,\HH)$ is called \emph{economical}. 

\begin{lemma}\label{lem:econ}
Suppose $(V,\HH)$ is a hypergraph with no isolated points. If $\HH$ is finite, then there is some $\HH' \sub \HH$ such that $(V,\HH')$ is economical and has no isolated points. 
\end{lemma}
\begin{proof}
Delete hyperedges from $\HH$, one by one, as long as it is possible to delete a hyperedge without creating any isolated points. When it is no longer possible to do so (which must be the case after finitely many steps), we have found $\HH'$.
\end{proof}

\begin{theorem}\label{thm:quadbound}
Suppose $(V,\HH)$ is a hypergraph without isolated points, and $|V| > n^2$. Then $(V,\HH)$ contains a partition of size greater than $n$.
\end{theorem}
\begin{proof}
If any $E \in \HH$ contains more than $n$ vertices, then taking $X = E$ and $\P = \{E\}$ shows that $(V,\HH)$ contains a partition of size $>\!n$.
So let us suppose that each $E \in \HH$ contains at most $n$ vertices. 

Suppose $V$ is finite. This implies $\HH$ is also finite (of size at most $2^{|V|}$). By the previous lemma, there is some $\HH' \sub \HH$ such that $(V,\HH')$ is economical. Because $\HH'$ is economical, we may for every $E \in \HH'$ find some $v_E \in V$ such that $v_E$ is isolated in $(V,\HH' \setminus \{E\})$. Then $\set{v_E}{E \in \HH'}$ and $\HH'$ form a partition in $(V,\HH)$ of size $|\HH'|$. As $(V,\HH)$ contains no isolated points, and $|V| > n^2$, and every $E \in \HH'$ contains at most $n$ vertices, we have $|\HH'| > n$. This finishes the proof for the case that $V$ is finite.

Now suppose $V$ is infinite. Pick some finite $W \sub V$ with $|W| > n^2$, and let $\HH_W = \set{W \cap E}{E \in \HH}$. By the previous paragraph, $(W,\HH_W)$ contains a partition of size $>\!n$, i.e., there is some $D \sub W$ with $|D| > n$ and some $\G \sub \HH_W$ such that each member of $D$ is contained in exactly one member of $\G$. For each $G \in \G$, choose some $E_G \in \HH$ such that $E_G \cap D = G$. Then $D$ and $\set{E_G}{G \in \G}$ form a partition of size $|D| > n$ in $(V,\HH)$.
\end{proof}

\begin{theorem}
$H(n)$ and $I(n)$ are well-defined for all $n \in \N$. Furthermore, $H(n) \leq n^2$ and $I(n) \leq n^2-2n+2$ for all $n$.
\end{theorem}
\begin{proof}
From Theorems~\ref{thm:Ser} and \ref{thm:quadbound} it follows that $H(n)$ is well-defined and $H(n) \leq n^2$ for all $n$, and that $I(n)$ is well-defined and 
$$I(n) \leq H(n-1)+1 \leq (n-1)^2+1 = n^2-2n+2$$ 
for all $n \geq 2$. To finish the proof, all that remains is an easy observation: $I(1)$ is well-defined, and $I(1) = 1$.
\end{proof}

Now that we know $\Ser(n)$ and $\Hyp(n)$ are well-defined, we proceed to sharpen our upper bound on their values. As we will see in the following section, the next theorem gives the right growth rate for $\Ser(n)$ and $\Hyp(n)$, up to a constant factor.

\begin{lemma}\label{lem:econ2}
$\Hyp(n)$ is equal to the greatest $k \in \N$ such that there is an economical hypergraph $(V,\HH)$ with $|V| = k$ containing no partitions of size greater than $n$.
\end{lemma}
\begin{proof}
By definition, there is a hypergraph $(V,\HH_0)$ with $|V| = \Hyp(n)$ containing no isolated points, and containing no partitions of size greater than $n$. By Lemma~\ref{lem:econ}, there is some $\HH \sub \HH_0$ such that $(V,\HH)$ is economical. If $X$ and $\P$ form a partition in $(V,\HH)$, then they also form a partition in $(V,\HH_0)$; hence $(V,\HH)$ contains no partitions of size greater than $n$. Thus the number $k$ described in the lemma is $\geq\! \Hyp(n)$. The reverse inequality follows immediately from the definition of $\Hyp(n)$.
\end{proof}

\begin{theorem}\label{thm:upper}
$\Hyp(n) \leq \sum_{k = 1}^n \frac{n}{k}$ for all $n \in \N$.
\end{theorem}
\begin{proof}
Fix $n \in \N$, and let $(V,\HH)$ be an economical hypergraph containing no partition of size greater than $n$. To prove the theorem, it suffices (by the previous lemma) to show $|V| \leq \sum_{k=1}^n \frac{n}{k}$.

Recall that a vertex $v \in V$ has \emph{degree} $d$ in $(V,\HH)$ if $\card{\set{E \in \HH}{v \in E}} = d$. For each $k \leq n$, define
$$D_k = \set{v \in V}{v \text{ has degree } k}.$$
More generally, if $\HH' \sub \HH$ then define
$$D_k^{\HH'} = \set{v \in V}{v \text{ has degree } k \text{ in } (V,\HH')}.$$

Observe that $D_1$ and $\HH$ form a partition in $(V,\HH)$. This implies $|D_1| \leq n$. Because $(V,\HH)$ is economical, there is an injection from $\HH$ to $V$. (For example, any injection that maps each $E \in \HH$ to some $v_E$ that is isolated in $(V,\HH \setminus \{E\})$.) Hence $|\HH| \leq |D_1| \leq n$. It follows that $(V,\HH)$ has no vertices of degree greater than $n$. As $(V,\HH)$ also has no isolated points,
$$V \,=\, \bigcup_{k = 1}^n D_k \qquad \text{ and } \qquad |V| = \sum_{k = 1}^n \card{D_k}.$$
For each $k \leq n$, define $m_k = k|D_k|$.

Fix $j$ with $0 \leq j < n$. In the next part of the proof, our goal will be to produce an inequality, labeled $(\mathrm{Ineq}_j)$ below, that constrains the values of the $m_k$ for $k \leq j+1$.

Let $\ell = n - |\HH|$ and let $F_1,F_2,\dots,F_\ell$ be any $\ell$ distinct sets that are not in $\HH$. Recall that $|\HH| \leq n$, so that $\ell \geq 0$. The $F_i$ will be used as dummy variables below. (We think of the $F_i$ as ``fake hyperedges'' whose purpose is to allow us to pretend that $|\HH| = n$, even if really $|\HH| < n$. In what follows, one gets the right idea by thinking of each $F_i$ as an empty edge.) 
Let
$$\mathbf{Delete}_j \,=\, \set{\A }{ \A \sub \HH \cup \set{F_i}{i \leq \ell} \text{ and } \card{\A} = n-j}.$$
Note that $\card{\mathbf{Delete}_j} = {n \choose j}$.

Roughly, our idea for obtaining an inequality constraining $m_1,\dots,m_{j+1}$ is as follows. Each $\A \in \mathbf{Delete}_j$ gives rise to a subset $\A \cap \HH$ of $\HH$. Observing that $\A \cap \HH$ and $D_1^{\A \cap \HH}$ form a partition in $(V,\HH)$, this means that $D_1^{\A \cap \HH}$ must have size $\leq\! n$ for each $\A$. Summing over all $\A \in \mathbf{Delete}_j$ will give the desired inequality.

Fix $k \leq n$ and $v \in D_k = D_k^{\HH}$. Given any $\HH' \sub \HH$, note that $v \in D_1^{\HH'}$ if and only if 
$\card{\set{E \in \HH'}{v \in E}} = 1,$
and this is the case if and only if
$$\card{\set{E \in \HH \setminus \HH'}{v \in E}} = k-1.$$
Therefore, given $\A \in \mathbf{Delete}_j$, $v \in D_1^{\A \cap \HH}$ if and only if the $j$ members of $(\HH \cup \set{F_i}{i \leq \ell}) \setminus \A$ consist of exactly $k-1$ members of the $k$-element set $\set{E \in \HH}{v \in E}$, plus any $j - (k-1)$ other members of the $n$-element set $\HH \cup \set{F_i}{i \leq \ell}$.
Hence, defining
$$\mathbf{S}_{v} \,=\, \set{ \A \in \mathbf{Delete}_j }{v \in D_1^{\A \cap \HH}},$$
we have 
$$\card{\mathbf{S}_v} \,=\, {k \choose k-1}{n-k \choose j - (k-1)} \,=\, k{n-k \choose j-k+1}$$
whenever $k \leq j+1$, and $\card{\mathbf{S}_v} = 0$ whenever $k > j+1$.
Note that $\card{\mathbf{S}_v}$ does not depend on $v$, but only on the degree $k$ of $v$ and on $j$.

By varying $k$ and $v$, and summing over all $\A \in \mathbf{Delete}_j$, we obtain
\begin{align*}
\sum_{\A \in \mathbf{Delete}_j} \card{D_1^{\HH \cap \A}} & \,=\, \sum_{v \in V} \card{\mathbf{S}_v} \,=\, \sum_{k = 1}^{n} \, \sum_{v \in D_k} \card{\mathbf{S}_v} \,=\, \sum_{k = 1}^{j+1} \, \sum_{v \in D_k} \card{\mathbf{S}_v} \\
& \,=\, \sum_{k=1}^{j+1} \, \sum_{v \in D_k} k{n-k \choose j-k+1} 
 \,=\, \sum_{k=1}^{j+1} \, \card{D_k} k{n-k \choose j-k+1} \\
& \,=\, \sum_{k=1}^{j+1} \, m_k{n-k \choose j-k+1}.
\end{align*}

Recall that for any $\HH' \sub \HH$, $D_1^{\HH'}$ and $\HH'$ form a partition in $(V,\HH)$. This implies $\card{D_1^{\HH'}} \leq n$ for all $\HH' \sub \HH$, and so
$$\sum_{\A \in \mathbf{Delete}_j} \card{D_1^{\HH \cap \A}} \,\leq\, \sum_{\A \in \mathbf{Delete}_j} \!\! n \,=\, n\card{\mathbf{Delete}_j} \,=\, n{n \choose j}.$$
Putting these observations together, we arrive at what we were aiming for, namely an inequality that constrains the $m_k$ for $k \leq j+1$:
\begin{equation}
\sum_{k=1}^{j+1} \, m_j{n-k \choose j-k+1} \,\leq\, n{n \choose j}.
\tag{Ineq$_j$}
\end{equation}

The next step in our proof is to take a positive linear combination of the inequalities $(\mathrm{Ineq}_0),(\mathrm{Ineq}_1),\dots,(\mathrm{Ineq}_{n-1})$ in order to obtain a single inequality. The coefficients for this linear combination come from taking the reciprocals of row $n-1$ of Pascal's triangle. That is, by taking the linear combination
$$\frac{1}{{n-1 \choose 0}}(\mathrm{Ineq}_0) + \frac{1}{{n-1 \choose 1}}(\mathrm{Ineq}_1) + \frac{1}{{n-1 \choose 2}}(\mathrm{Ineq}_2) + \dots + \frac{1}{{n-1 \choose n-1}}(\mathrm{Ineq}_{n-1}),$$
we arrive at a new inequality combining all the $(\mathrm{Ineq}_j)$:
\begin{equation}
\sum_{j = 0}^{n-1} \, \frac{\sum_{k=1}^{j+1} \, m_k{n-k \choose j-k+1} }{{n-1 \choose j}} \ \leq \ \sum_{j = 0}^{n-1} \frac{n{n \choose j}}{{n-1 \choose j}}
\tag{$\star$}
\end{equation}

While it is far from obvious at this point, we shall see that $(\star)$ simplifies to the inequality claimed in the statement of the theorem.
The following claim shows how to simplify the left-hand side of $(*)$.

\begin{claim} $\displaystyle \sum_{j = 0}^{n-1} \, \frac{\sum_{k=1}^{j+1} \, m_k{n-k \choose j-k+1} }{{n-1 \choose j}} \ = \ \sum_{k = 1}^n \frac{nm_k}{k}$.
\end{claim}
\begin{proof}[Proof of Claim]
Using a Fubini-like trick to rearrange the sum on the left, we obtain
\begin{align*}
\sum_{j = 0}^{n-1} \, \frac{\sum_{k=1}^{j+1} \, m_k{n-k \choose j-k+1} }{{n-1 \choose j}} & \,=\, \sum_{j = 0}^{n-1} \, \sum_{k=1}^{j+1} \, m_k \frac{ {n-k \choose j-k+1} }{{n-1 \choose j}} \,=\, \sum_{0 \,<\, k \,\leq\, j+1 \,\leq\, n} m_k \frac{ {n-k \choose j-k+1} }{{n-1 \choose j}} \\
& \,=\, \sum_{k=1}^n \, \sum_{j = k-1}^{n-1} m_k \frac{ {n-k \choose j-k+1} }{{n-1 \choose j}} \,=\, \sum_{k=1}^n \, m_k \sum_{j = k-1}^{n-1} \frac{ {n-k \choose j-k+1} }{{n-1 \choose j}}
\end{align*}
Thus, to prove the claim, it suffices to show that
\begin{equation}
\sum_{j = k-1}^{n-1} \frac{ {n-k \choose j-k+1} }{{n-1 \choose j}} \ = \ \frac{n}{k}
\tag{$\ddagger$}
\end{equation}
whenever $1 \leq k \leq n$. 

To see this, first recall that
$${m \choose r}{r \choose s} \,=\, {m \choose s}{m-s \choose r-s} \qquad \text{ for all } m,r,s \in \N \text{ with } s \leq r \leq m.$$
(See, e.g., \cite[Eq. 5.21]{GKP}.) Setting $m=n-1$, $r=j$, $s=k-1$ gives
$${n-1 \choose j}{j \choose k-1} \,=\, {n-1 \choose k-1}{n-k \choose j-k+1},$$ 
or equivalently
$$\frac{ {n-k \choose j-k+1} }{{n-1 \choose j}} \,=\, \frac{ {j \choose k-1} }{{n-1 \choose k-1}}.$$


\noindent Substituting this into the left-hand side of $(\ddagger)$ yields
$$\sum_{j = k-1}^{n-1} \frac{ {n-k \choose j-k+1} }{{n-1 \choose j}} \,=\, \sum_{j = k-1}^{n-1} \frac{ {j \choose k-1} }{{n-1 \choose k-1}} \,=\, \frac{1}{{n-1 \choose k-1}} \sum_{j = k-1}^{n-1} {j \choose k-1}.$$
Now recall the well-known \emph{hockey stick identity} \cite{Jones}, also sometimes known as the \emph{Christmas stocking identity}, which states that
$$\sum_{j=r}^m {j \choose r} \,=\, {m+1 \choose r+1} \qquad \text{ for all } m,r \in \N \text{ with } r \leq m.$$
Plugging in $r = k-1$ and $m = n-1$, we get
$$\sum_{j = k-1}^{n-1} {j \choose k-1} \,=\, {n \choose k},$$
and combining this with our earlier observations gives
$$\sum_{j = k-1}^{n-1} \frac{ {n-k \choose j-k+1} }{{n-1 \choose j}} \,=\, \frac{1}{{n-1 \choose k-1}} \sum_{j = k-1}^{n-1} {j \choose k-1} \,=\, \frac{ {n \choose k} }{ {n-1 \choose k-1} } \,=\, \frac{n}{k},$$
which proves $(\ddagger)$ and finishes the proof of the claim.
\end{proof}

Returning to the inequality $(\star)$, and applying the preceding claim, we obtain
$$\sum_{k = 1}^n \frac{nm_k}{k} \,\leq\, \sum_{j = 0}^{n-1} \frac{n{n \choose j}}{{n-1 \choose j}}.$$
To simplify the right-hand side, observe that
$$\sum_{j = 0}^{n-1} \frac{n{n \choose j}}{{n-1 \choose j}} \,=\, n \sum_{j=0}^{n-1} \frac{{n \choose j}}{{n-1 \choose j}} \,=\, n \sum_{j=0}^{n-1} \frac{\frac{n!}{(n-j)!j!}}{\quad \frac{(n-1)!}{(n-j-1)!j!}\quad} \,=\, n \sum_{j = 0}^{n-1} \frac{n}{n-j}.$$
Substituting $k = n-j$ and reversing the order of the summation,
$$n \sum_{j = 0}^{n-1} \frac{n}{n-j} \,=\, n \sum_{k=1}^n \frac{n}{k}.$$
Hence
$$\sum_{k = 1}^n \frac{nm_k}{k} \,\leq\, n \sum_{k=1}^n \frac{n}{k}$$
and dividing both sides by $n$ gives
$$\sum_{k = 1}^n \frac{m_k}{k} \,\leq\, \sum_{k=1}^n \frac{n}{k}.$$
But recall the definition of $m_k$, namely $m_k = k\card{D_k}$, where $D_k$ denotes the number of vertices in $(V,\HH)$ of degree $k$. As every vertex has degree at least $1$ (because there are no isolated points) and at most $n$ (because $|\HH| \leq n$),
$$|V| \,=\, \sum_{k = 1}^n \card{D_k} \,=\, \sum_{k = 1}^n \frac{m_k}{k} \,\leq\, \sum_{k=1}^n \frac{n}{k}$$
as claimed.
\end{proof}

The following corollary simply applies well-known results and techniques to rephrase the conclusion of the previous theorem in a way that underscores the asymptotic growth rates of $\Hyp(n)$ and $\Ser(n)$.

\begin{corollary}\label{cor:ubSer}
For every $n \in \N$,
$$\textstyle \Hyp(n) \,<\, n \ln n + \g n + \frac{1}{2},$$
$$\Ser(n) \,<\, \textstyle n \ln n + \g n - \ln n +\frac{3}{2} -\g,$$
where $\g \approx .5772156649$ is the Euler-Mascheroni constant.
\end{corollary}
\begin{proof}
Let $H_n = 1+\frac{1}{2}+\frac{1}{3}+\dots+\frac{1}{n}$ denote the $n^{\mathrm{th}}$ harmonic number. Bounds for $H_n$ are given in \cite{Conway&Guy,Havil} (among other places), namely 
$H_n \,<\, \ln n + \g + \frac{1}{2n}$.
The first assertion of the corollary follows immediately from this and the previous theorem:
$$\textstyle \Hyp(n) \,\leq\, nH_n \,<\, n \ln n + \g n + \frac{1}{2}.$$

For the second assertion, Theorem~\ref{thm:Ser} and the previous paragraph combine to give
\begin{align*}
\Ser(n) & \textstyle \,\leq\, n\Hyp(n-1)+1 \,<\, (n-1) \ln (n-1) + \g (n-1) + \frac{3}{2} \\
& \textstyle \,\leq\, (n-1) \ln n + \g n - \g + \frac{3}{2} \,=\, n \ln n + \g n - \ln n +\frac{3}{2} - \g,
\end{align*}
which finishes the proof.
\end{proof}

The bound $(n-1) \ln (n-1) + \g (n-1) + \frac{3}{2} \leq n \ln n + \g n - \ln n +\frac{3}{2} - \g$ used in the proof has optimal coefficients on the right-hand side for the first three terms, but the constant term is asymptotically too large by $1$. With a little more work one can obtain
$$\textstyle \Ser(n) \,<\, (n-1) \ln (n-1) + \g (n-1) + \frac{3}{2} \,\leq\, n \ln n + \g n - \ln n + \frac{1}{2} - \g + \frac{1}{2n-1}$$
for all $n \in \N$.

\section{Lower bounds for $\Ser(n)$ and $\Hyp(n)$}\label{sec:combinatorics2}

We now move on to the task of finding a lower bound for $\Ser(n)$.

\begin{definition}\label{def:Serbounding}
Given $n,k \in \N$, a subset $\F$ of $\mathrm{Fn}(k,2)$ is called \emph{$\Ser(n)$-bounding} if it is full, and if for every $\G \sub \F$ and $D \sub \bigcup_{g \in \G}\mathrm{dom}(g)$ with $|D| \geq n$, there are two functions in $\G$ that disagree at a point of $D$.
\end{definition}

Observe that, for all $n \in \N$, 
\begin{align*}
\Ser(n) & = \min \set{k}{\text{there is no } \Ser(n)\text{-bounding family of size } \!\geq\! k} \\
& = \textstyle \max \set{\card{\bigcup \F}}{\F \text{ is }\Ser(n)\text{-bounding}} + 1
\end{align*}

\begin{lemma}\label{lem:lbSer}
Define an infinite sequence $k_1, k_2, k_3, \dots$ of natural numbers via the following recurrence relation:
$$k_1 = 1 \qquad \text{ and } \qquad k_n \,=\, \left \lfloor \frac{n}{2} \right \rfloor + k_{\left \lfloor \frac{n}{2} \right \rfloor} + k_{\left \lfloor \frac{n+1}{2} \right \rfloor}.$$
For every $n$, there is an $\Ser(n+1)$-bounding full subset of $\mathrm{Fn}(k_n,2)$.
\end{lemma}
\begin{proof}
We begin the proof by constructing a sequence $T_1,T_2,T_3,\dots$ of rooted trees. Afterward, these trees will be used to construct the desired families of functions.
Recall that every rooted tree comes equipped with a natural partial order: $v \leq w$ if and only if the unique path from the root to $w$ contains $u$. In what follows, we move freely between the notion of a rooted tree as a particular type of pointed graph, and the notion of a rooted tree as a particular type of partial order.

The construction of the $T_n$ is by recursion. To begin, let $T_1$ be the rooted tree with exactly one vertex. Given $T_1,T_2,\dots,T_{n-1}$, the rooted tree $T_n$ is defined so that
\begin{itemize}
\item[$\circ$] the bottom of $T_n$ consists of $\left \lfloor \frac{n}{2} \right \rfloor$ linearly ordered vertices, the bottommost one being the root of $T_n$.
\item[$\circ$] the topmost of these $\left \lfloor \frac{n}{2} \right \rfloor$ linearly ordered vertices has two vertices immediately above it; one is the root of an isomorphic copy of $T_{\left \lfloor \frac{n}{2} \right \rfloor}$, and the other is the root of an isomorphic copy of $T_{\left \lfloor \frac{n+1}{2} \right \rfloor}$.
\end{itemize}

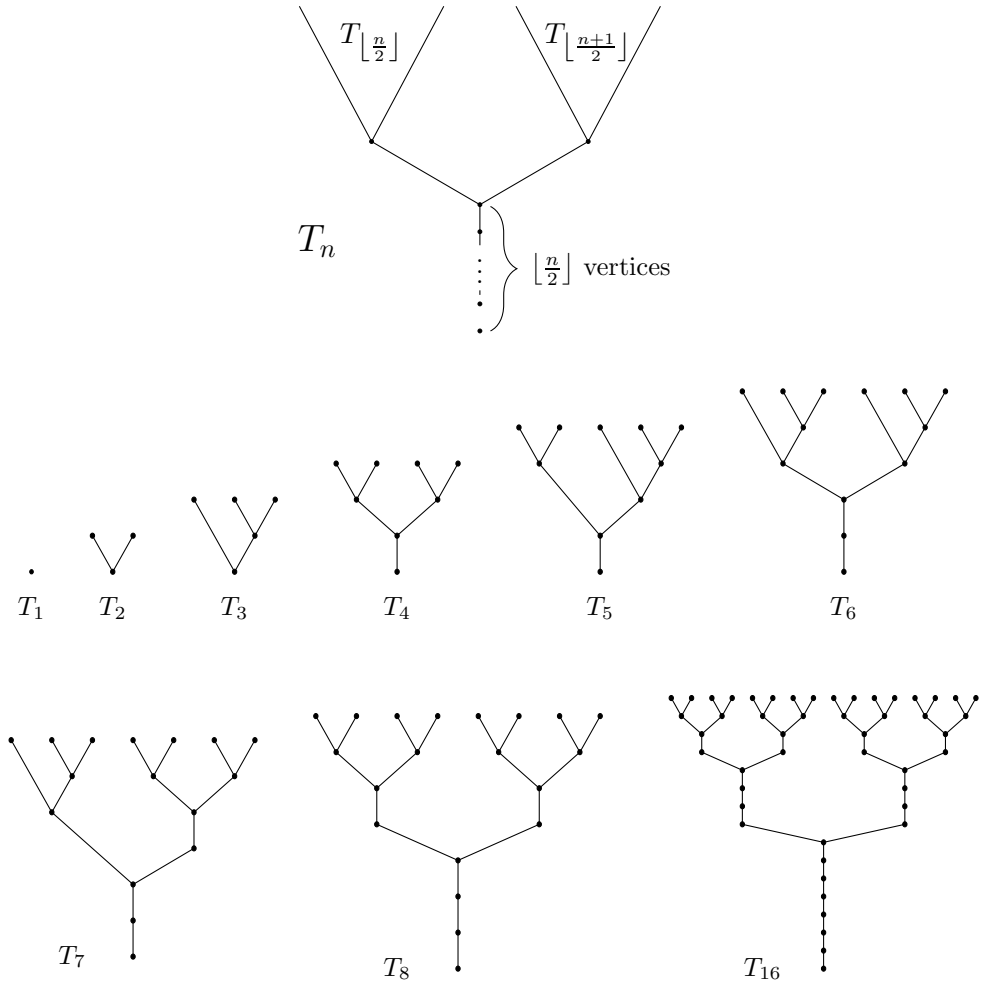
\begin{center}
\begin{tikzpicture}[xscale=.24,yscale=.24]

\node at (7,-47) {\Large $T_n$};
\draw[fill=black] (16,-52) circle (3pt);
\draw[fill=black] (16,-50.5) circle (3pt);
\draw (16,-50) -- (16,-49.75);
\node at (16,-48.25) {\small $\vdots$};
\draw[fill=black] (16,-46.5) circle (3pt);
\draw[fill=black] (16,-45) circle (3pt);
\draw (16,-47.25) -- (16,-45);
\draw (16,-45) -- (10,-41.5);
\draw (16,-45) -- (22,-41.5);
\draw [decorate,decoration={brace,amplitude=10pt,mirror,raise=4pt},yshift=0pt] (16,-52) -- (16,-45.1) 
node [black,midway,xshift=1.6cm] {\small $\left\lfloor \frac{n}{2} \right\rfloor$ vertices};

\draw[fill=black] (10,-41.5) circle (3pt);
\draw[fill=black] (22,-41.5) circle (3pt);
\draw (10,-41.5) -- (6,-34);
\draw (10,-41.5) -- (14,-34);
\draw (22,-41.5) -- (18,-34);
\draw (22,-41.5) -- (26,-34);
\node at (10,-36) { $T_{\left\lfloor \frac{n}{2} \right\rfloor}$};
\node at (22,-36) { $T_{\left\lfloor \frac{n+1}{2} \right\rfloor}$};

\end{tikzpicture}
\end{center}


\vspace{-2mm}
\begin{center}
\begin{tikzpicture}[xscale=.27,yscale=.32]

\draw[fill=black] (0,0) circle (2.5pt);
\node at (0,-1.5) {\small $T_1$};

\draw[fill=black] (4,0) circle (3pt);
\draw[fill=black] (3,1.5) circle (3pt);
\draw[fill=black] (5,1.5) circle (3pt);
\draw (4,0) -- (3,1.5);
\draw (4,0) -- (5,1.5);
\node at (4,-1.5) {\small $T_2$};

\draw[fill=black] (10,0) circle (3pt);
\draw[fill=black] (11,1.5) circle (3pt);
\draw[fill=black] (10,3) circle (3pt);
\draw[fill=black] (12,3) circle (3pt);
\draw[fill=black] (8,3) circle (3pt);
\draw (10,0) -- (11,1.5);
\draw (10,0) -- (8,3);
\draw (11,1.5) -- (10,3);
\draw (11,1.5) -- (12,3);
\node at (10,-1.5) {\small $T_3$};

\draw[fill=black] (18,0) circle (3pt);
\draw[fill=black] (18,1.5) circle (3pt);
\draw[fill=black] (16,3) circle (3pt);
\draw[fill=black] (20,3) circle (3pt);
\draw[fill=black] (15,4.5) circle (3pt);
\draw[fill=black] (17,4.5) circle (3pt);
\draw[fill=black] (19,4.5) circle (3pt);
\draw[fill=black] (21,4.5) circle (3pt);
\draw (18,0) -- (18,1.5);
\draw (18,1.5) -- (16,3);
\draw (18,1.5) -- (20,3);
\draw (16,3) -- (15,4.5);
\draw (16,3) -- (17,4.5);
\draw (20,3) -- (19,4.5);
\draw (20,3) -- (21,4.5);
\node at (18,-1.5) {\small $T_4$};

\draw[fill=black] (28,0) circle (3pt);
\draw[fill=black] (28,1.5) circle (3pt);
\draw[fill=black] (30,3) circle (3pt);
\draw[fill=black] (25,4.5) circle (3pt);
\draw[fill=black] (31,4.5) circle (3pt);
\draw[fill=black] (24,6) circle (3pt);
\draw[fill=black] (26,6) circle (3pt);
\draw[fill=black] (28,6) circle (3pt);
\draw[fill=black] (30,6) circle (3pt);
\draw[fill=black] (32,6) circle (3pt);
\draw (28,0) -- (28,1.5);
\draw (28,1.5) -- (30,3);
\draw (28,1.5) -- (25,4.5);
\draw (25,4.5) -- (24,6);
\draw (25,4.5) -- (26,6);
\draw (30,3) -- (31,4.5);
\draw (30,3) -- (28,6);
\draw (31,4.5) -- (30,6);
\draw (31,4.5) -- (32,6);
\node at (28,-1.5) {\small $T_5$};

\begin{scope}[shift={(33,14)}]

\draw[fill=black] (7,-14) circle (3pt);
\draw[fill=black] (7,-12.5) circle (3pt);
\draw[fill=black] (7,-11) circle (3pt);
\draw[fill=black] (4,-9.5) circle (3pt);
\draw[fill=black] (10,-9.5) circle (3pt);
\draw[fill=black] (5,-8) circle (3pt);
\draw[fill=black] (11,-8) circle (3pt);
\draw[fill=black] (2,-6.5) circle (3pt);
\draw[fill=black] (4,-6.5) circle (3pt);
\draw[fill=black] (6,-6.5) circle (3pt);
\draw[fill=black] (8,-6.5) circle (3pt);
\draw[fill=black] (10,-6.5) circle (3pt);
\draw[fill=black] (12,-6.5) circle (3pt);
\draw (7,-14) -- (7,-12.5);
\draw (7,-12.5) -- (7,-11);
\draw (7,-11) -- (4,-9.5);
\draw (7,-11) -- (10,-9.5);
\draw (4,-9.5) -- (5,-8);
\draw (10,-9.5) -- (11,-8);
\draw (4,-9.5) -- (2,-6.5);
\draw (10,-9.5) -- (8,-6.5);
\draw (5,-8) -- (4,-6.5);
\draw (5,-8) -- (6,-6.5);
\draw (11,-8) -- (10,-6.5);
\draw (11,-8) -- (12,-6.5);
\node at (7,-15.5) {\small $T_6$};

\end{scope}

\begin{scope}[shift={(-19,-2)}]
\draw[fill=black] (24,-14) circle (3pt);
\draw[fill=black] (24,-12.5) circle (3pt);
\draw[fill=black] (24,-11) circle (3pt);
\draw[fill=black] (27,-9.5) circle (3pt);
\draw[fill=black] (20,-8) circle (3pt);
\draw[fill=black] (21,-6.5) circle (3pt);
\draw[fill=black] (27,-8) circle (3pt);
\draw[fill=black] (25,-6.5) circle (3pt);
\draw[fill=black] (29,-6.5) circle (3pt);
\draw[fill=black] (18,-5) circle (3pt);
\draw[fill=black] (20,-5) circle (3pt);
\draw[fill=black] (22,-5) circle (3pt);
\draw[fill=black] (24,-5) circle (3pt);
\draw[fill=black] (26,-5) circle (3pt);
\draw[fill=black] (28,-5) circle (3pt);
\draw[fill=black] (30,-5) circle (3pt);
\draw (24,-14) -- (24,-12.5);
\draw (24,-12.5) -- (24,-11);
\draw (24,-11) -- (27,-9.5);
\draw (24,-11) -- (20,-8);
\draw (20,-8) -- (18,-5);
\draw (20,-8) -- (21,-6.5);
\draw (21,-6.5) -- (20,-5);
\draw (21,-6.5) -- (22,-5);
\draw (27,-9.5) -- (27,-8);
\draw (27,-8) -- (25,-6.5);
\draw (25,-6.5) -- (24,-5);
\draw (25,-6.5) -- (26,-5);
\draw (27,-8) -- (29,-6.5);
\draw (29,-6.5) -- (28,-5);
\draw (29,-6.5) -- (30,-5);
\node at (21,-14) {\small $T_7$};
\end{scope}

\begin{scope}[shift={(14,13)}]

\draw[fill=black] (0,-19) circle (3pt);
\draw[fill=black] (2,-19) circle (3pt);
\draw[fill=black] (4,-19) circle (3pt);
\draw[fill=black] (6,-19) circle (3pt);
\draw[fill=black] (8,-19) circle (3pt);
\draw[fill=black] (10,-19) circle (3pt);
\draw[fill=black] (12,-19) circle (3pt);
\draw[fill=black] (14,-19) circle (3pt);
\draw[fill=black] (1,-20.5) circle (3pt);
\draw[fill=black] (5,-20.5) circle (3pt);
\draw[fill=black] (9,-20.5) circle (3pt);
\draw[fill=black] (13,-20.5) circle (3pt);
\draw[fill=black] (3,-22) circle (3pt);
\draw[fill=black] (11,-22) circle (3pt);
\draw[fill=black] (3,-23.5) circle (3pt);
\draw[fill=black] (11,-23.5) circle (3pt);
\draw[fill=black] (7,-25) circle (3pt);
\draw[fill=black] (7,-26.5) circle (3pt);
\draw[fill=black] (7,-28) circle (3pt);
\draw[fill=black] (7,-29.5) circle (3pt);
\draw (7,-29.5) -- (7,-28);
\draw (7,-28) -- (7,-26.5);
\draw (7,-26.5) -- (7,-25);
\draw (7,-25) -- (3,-23.5);
\draw (3,-23.5) -- (3,-22);
\draw (3,-22) -- (1,-20.5);
\draw (1,-20.5) -- (0,-19);
\draw (1,-20.5) -- (2,-19);
\draw (3,-22) -- (5,-20.5);
\draw (5,-20.5) -- (4,-19);
\draw (5,-20.5) -- (6,-19);
\draw (7,-25) -- (11,-23.5);
\draw (11,-23.5) -- (11,-22);
\draw (11,-22) -- (9,-20.5);
\draw (9,-20.5) -- (8,-19);
\draw (9,-20.5) -- (10,-19);
\draw (11,-22) -- (13,-20.5);
\draw (13,-20.5) -- (12,-19);
\draw (13,-20.5) -- (14,-19);
\node at (4,-29.5) {\small $T_8$};

\draw[fill=black] (18,-19) circle (3pt);
\draw[fill=black] (20,-19) circle (3pt);
\draw[fill=black] (22,-19) circle (3pt);
\draw[fill=black] (24,-19) circle (3pt);
\draw[fill=black] (26,-19) circle (3pt);
\draw[fill=black] (28,-19) circle (3pt);
\draw[fill=black] (30,-19) circle (3pt);
\draw[fill=black] (32,-19) circle (3pt);
\draw[fill=black] (19,-20.5) circle (3pt);
\draw[fill=black] (23,-20.5) circle (3pt);
\draw[fill=black] (27,-20.5) circle (3pt);
\draw[fill=black] (31,-20.5) circle (3pt);
\draw[fill=black] (21,-22) circle (3pt);
\draw[fill=black] (29,-22) circle (3pt);
\draw[fill=black] (21,-23.5) circle (3pt);
\draw[fill=black] (29,-23.5) circle (3pt);
\draw[fill=black] (25,-25) circle (3pt);
\draw[fill=black] (25,-26.5) circle (3pt);
\draw[fill=black] (25,-28) circle (3pt);
\draw[fill=black] (25,-29.5) circle (3pt);
\draw[fill=black] (17.5,-18.25) circle (3pt);
\draw[fill=black] (18.5,-18.25) circle (3pt);
\draw[fill=black] (19.5,-18.25) circle (3pt);
\draw[fill=black] (20.5,-18.25) circle (3pt);
\draw[fill=black] (21.5,-18.25) circle (3pt);
\draw[fill=black] (22.5,-18.25) circle (3pt);
\draw[fill=black] (23.5,-18.25) circle (3pt);
\draw[fill=black] (24.5,-18.25) circle (3pt);
\draw[fill=black] (25.5,-18.25) circle (3pt);
\draw[fill=black] (26.5,-18.25) circle (3pt);
\draw[fill=black] (27.5,-18.25) circle (3pt);
\draw[fill=black] (28.5,-18.25) circle (3pt);
\draw[fill=black] (29.5,-18.25) circle (3pt);
\draw[fill=black] (30.5,-18.25) circle (3pt);
\draw[fill=black] (31.5,-18.25) circle (3pt);
\draw[fill=black] (32.5,-18.25) circle (3pt);
\draw[fill=black] (19,-19.75) circle (3pt);
\draw[fill=black] (23,-19.75) circle (3pt);
\draw[fill=black] (27,-19.75) circle (3pt);
\draw[fill=black] (31,-19.75) circle (3pt);
\draw[fill=black] (21,-21.25) circle (3pt);
\draw[fill=black] (29,-21.25) circle (3pt);
\draw[fill=black] (21,-22.75) circle (3pt);
\draw[fill=black] (29,-22.75) circle (3pt);
\draw[fill=black] (25,-24.25) circle (3pt);
\draw[fill=black] (25,-25.75) circle (3pt);
\draw[fill=black] (25,-27.25) circle (3pt);
\draw[fill=black] (25,-28.75) circle (3pt);
\draw (25,-29.5) -- (25,-24.25);
\draw (25,-24.25) -- (21,-23.5);
\draw (25,-24.25) -- (29,-23.5);
\draw (21,-23.5) -- (21,-21.25);
\draw (29,-23.5) -- (29,-21.25);
\draw (21,-21.25) -- (19,-20.5);
\draw (19,-20.5) -- (19,-19.75);
\draw (19,-19.75) -- (18,-19);
\draw (19,-19.75) -- (20,-19);
\draw (18,-19) -- (17.5,-18.25);
\draw (18,-19) -- (18.5,-18.25);
\draw (20,-19) -- (19.5,-18.25);
\draw (20,-19) -- (20.5,-18.25);
\draw (23,-20.5) -- (23,-19.75);
\draw (23,-19.75) -- (22,-19);
\draw (23,-19.75) -- (24,-19);
\draw (22,-19) -- (21.5,-18.25);
\draw (22,-19) -- (22.5,-18.25);
\draw (24,-19) -- (23.5,-18.25);
\draw (24,-19) -- (24.5,-18.25);
\draw (21,-21.25) -- (23,-20.5);
\draw (29,-21.25) -- (27,-20.5);
\draw (27,-20.5) -- (27,-19.75);
\draw (27,-19.75) -- (26,-19);
\draw (27,-19.75) -- (28,-19);
\draw (26,-19) -- (25.5,-18.25);
\draw (26,-19) -- (26.5,-18.25);
\draw (28,-19) -- (27.5,-18.25);
\draw (28,-19) -- (28.5,-18.25);
\draw (31,-20.5) -- (31,-19.75);
\draw (31,-19.75) -- (30,-19);
\draw (31,-19.75) -- (32,-19);
\draw (30,-19) -- (29.5,-18.25);
\draw (30,-19) -- (30.5,-18.25);
\draw (32,-19) -- (31.5,-18.25);
\draw (32,-19) -- (32.5,-18.25);
\draw (29,-21.25) -- (31,-20.5);
\node at (22,-29.5) {\small $T_{16}$};

\end{scope}

\end{tikzpicture}
\end{center}

Let us say that a vertex $v \in V(T_n)$ is \emph{splitting} if it has more than one immediate successor. The following claim is fairly obvious, but will be useful in what follows:

\begin{claim}
For each $n \in \N$, every $v \in V(T_n)$ that is not $\leq$-maximal either is a splitting vertex, or else it has a splitting vertex above it. Also, every splitting vertex in $V(T_n)$ has exactly two successors.
\end{claim}
\begin{proof}[Proof of claim:]
Both assertions are easily proved by induction on $n$.
\end{proof}

For every non-maximal vertex $v$ of $T_n$, let $\s(v)$ denote the $\leq$-least splitting vertex $w$ such that $v \leq w$. Some such vertex exists by the previous claim.

Recall that two vertices $b$ and $c$ in a partial order are \emph{incomparable} if $b \not\leq c$ and $c \not\leq b$. If $b$ and $c$ are incomparable, then let us write $b \wedge c$ to denote the $\leq$-greatest vertex that is below both $b$ and $c$. (Note that $b \wedge c$ is well-defined, because at least one vertex must be below both $b$ and $c$, namely the root.)

\begin{claim}
Fix $n \in \N$, and let $D \sub V(T_n)$ with $|D| > n$. Then there exist $a,b,c \in D$ such that $b$ and $c$ are incomparable, and $\s(a) = b \wedge c$.
\end{claim}
\begin{proof}[Proof of claim:]
The proof is by induction on $n$. 

The base case $n = 1$ is vacuously true: $T_1$ has only a single vertex, so there is no $D \sub V(T_1)$ with $|D| > 1$. 

Let $n > 1$ and suppose the claim holds for all $m < n$. (In fact, we really only need the inductive hypothesis for $m = \left \lfloor \frac{n}{2} \right \rfloor, \left \lfloor \frac{n+1}{2} \right \rfloor$; observe that both these values of $m$ are strictly less than $n$ when $n > 1$.) For convenience let $A$, $B$, and $C$ denote, respectively, the $\left \lfloor \frac{n}{2} \right \rfloor$ vertices at the bottom of $T_n$, the $\card{V(T_{\left \lfloor \frac{n}{2} \right \rfloor})}$ vertices forming an isomorphic copy of $T_{\left \lfloor \frac{n}{2} \right \rfloor}$, and the $\card{V(T_{\left \lfloor \frac{n+1}{2} \right \rfloor})}$ vertices forming an isomorphic copy of $T_{\left \lfloor \frac{n+1}{2} \right \rfloor}$.

Let $D \sub V(T_n)$ with $|D| > n$. If $|D \cap B| > \left \lfloor \frac{n}{2} \right \rfloor$ then, because the claim holds for $\left \lfloor \frac{n}{2} \right \rfloor$ by hypothesis and the vertices in $B$ form an isomorphic copy of $T_{\left \lfloor \frac{n}{2} \right \rfloor}$, $D \cap B$ must contain some $a,b,c$ such that $b$ and $c$ are incomparable and $\s(a) = b \wedge c$. Similarly, if $|D \cap C| > \left \lfloor \frac{n+1}{2} \right \rfloor$ then, because the claim holds for $\left \lfloor \frac{n+1}{2} \right \rfloor$ by hypothesis, $D \cap C$ must contain some $a,b,c$ such that $b$ and $c$ are incomparable and $\s(a) = b \wedge c$.
For the remaining case, suppose that $|D \cap B| \leq \left \lfloor \frac{n}{2} \right \rfloor$ and $|D \cap C| \leq \left \lfloor \frac{n+1}{2} \right \rfloor$. Observe that 
$$\textstyle |D \cap (B \cup C)| = |D \cap B| + |D \cap C| \leq \left \lfloor \frac{n}{2} \right \rfloor + \left \lfloor \frac{n+1}{2} \right \rfloor = n$$
and because $|D| > n$, this implies $D \cap A \neq \0$. Similarly,
$$\textstyle |D \cap (A \cup B)| \leq |A| + |D \cap B| \leq \left \lfloor \frac{n}{2} \right \rfloor + \left \lfloor \frac{n}{2} \right \rfloor \leq n$$
and because $|D| > n$, this implies $D \cap C \neq \0$. Similarly,
$$\textstyle |D \cap (A \cup C)| \leq |A| + |D \cap C| \leq \left \lfloor \frac{n}{2} \right \rfloor + \left \lfloor \frac{n+1}{2} \right \rfloor = n$$
and because $|D| > n$, this implies $D \cap B \neq \0$. Thus each of $D \cap A$, $D \cap B$, and $D \cap C$ is nonempty. It is clear that if we choose any $a \in D \cap A$, $b \in D \cap B$, and $c \in D \cap C$, then $b$ and $c$ are incomparable, and $\s(a) = b \wedge c$. Thus the claim is true for $n$ and, by induction, the claim is true for all $n \in \N$.
\end{proof}

Recall that a \emph{maximal chain} in one of the $T_n$ is a set of vertices of $T_n$ that is $(1)$ linearly ordered by $\leq$, and $(2)$ properly contained in no other linearly ordered set of vertices. Equivalently, a maximal chain is (the underlying set of) a path in $T_n$ connecting the root to a $\leq$-maximal vertex.

\begin{claim}
For every $n \in \N$, there is an $\Ser(n+1)$-bounding family $\F \sub \mathrm{Fn}(k,2)$, where $k = \card{V(T_n)}$.
\end{claim}
\begin{proof}[Proof of claim:]
Fix $n \in \N$, and let $\leq$ denote the natural tree order on $V(T_n)$. Every splitting vertex in $V(T_n)$ has exactly two successors by a previous claim; for every splitting vertex $s \in V(T_n)$, arbitrarily define a bijection $\lambda_s$ from the two successors of $s$ to the set $\{p,n\}$. For each maximal chain $P \sub V(T_n)$, if $v$ is not the topmost vertex of $P$ then let $\mathrm{succ}_P(v)$ denote the vertex in $P$ immediately above $v$.

To each maximal chain $P$ in $T_n$ we now associate two functions, denoted $f^+_P$ and $f^-_P$. The domain of both functions is $P$. If $v \in P$ is not the topmost vertex of $P$, then define
$$f^+_P(v) \,=\, f^-_P(v) \,=\, \lambda_{\s(v)}(\mathrm{succ}_P(\s(v)))$$
and if $v$ is the topmost vertex in $P$, then let $f^+_P(t) = p$ and $f^-_P(t) = n$.

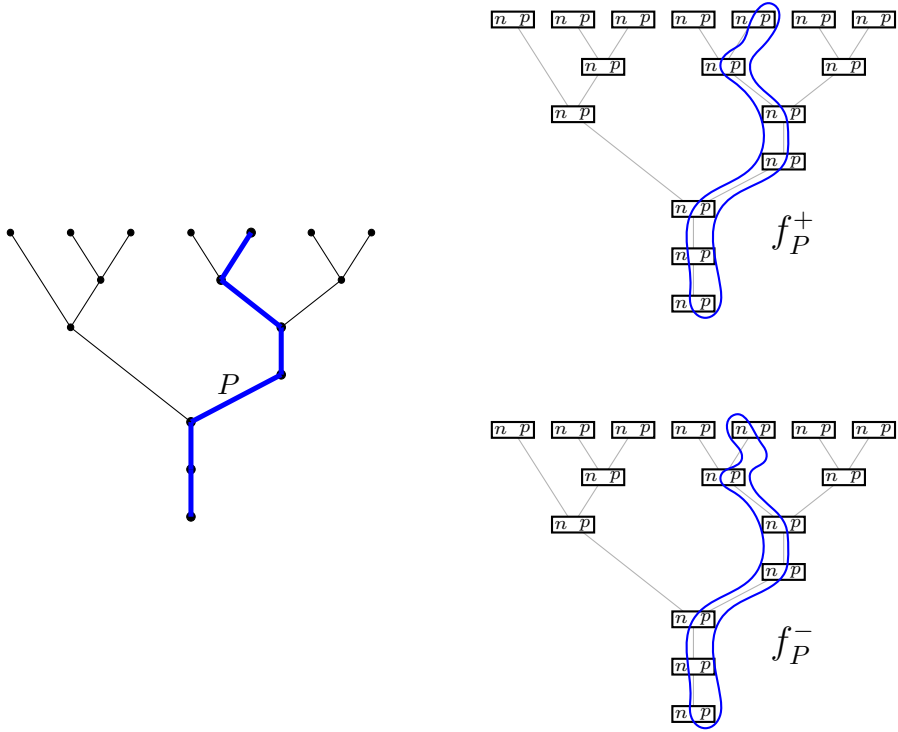
\begin{center}
\begin{tikzpicture}[xscale=.4,yscale=.42]

\node at (5.25,-3.3) {$P$};

\draw[fill=black] (4,-7.5) circle (4pt);
\draw[fill=black] (4,-6) circle (4pt);
\draw[fill=black] (4,-4.5) circle (4pt);
\draw[fill=black] (7,-3) circle (4pt);
\draw[fill=black] (0,-1.5) circle (3pt);
\draw[fill=black] (1,0) circle (3pt);
\draw[fill=black] (7,-1.5) circle (4pt);
\draw[fill=black] (5,0) circle (4pt);
\draw[fill=black] (9,0) circle (3pt);
\draw[fill=black] (-2,1.5) circle (3pt);
\draw[fill=black] (0,1.5) circle (3pt);
\draw[fill=black] (2,1.5) circle (3pt);
\draw[fill=black] (4,1.5) circle (3pt);
\draw[fill=black] (6,1.5) circle (4pt);
\draw[fill=black] (8,1.5) circle (3pt);
\draw[fill=black] (10,1.5) circle (3pt);
\draw[blue,line width=.67mm] (4,-7.5) -- (4,-6);
\draw[blue,line width=.67mm] (4,-6) -- (4,-4.5);
\draw[blue,line width=.67mm] (4,-4.5) -- (7,-3);
\draw (4,-4.5) -- (0,-1.5);
\draw (0,-1.5) -- (-2,1.5);
\draw (0,-1.5) -- (1,0);
\draw (1,0) -- (0,1.5);
\draw (1,0) -- (2,1.5);
\draw[blue,line width=.67mm] (7,-3) -- (7,-1.5);
\draw[blue,line width=.67mm] (7,-1.5) -- (5,0);
\draw (5,0) -- (4,1.5);
\draw[blue,line width=.67mm] (5,0) -- (6,1.5);
\draw (7,-1.5) -- (9,0);
\draw (9,0) -- (8,1.5);
\draw (9,0) -- (10,1.5);

\draw[black!30] (20.7,-13.75) -- (20.7,-10.75);
\draw[black!30] (20.7,-10.75) -- (23.7,-9.25);
\draw[black!30] (20.7,-10.75) -- (16.7,-7.75);
\draw[black!30] (16.7,-7.75) -- (17.7,-6.25);
\draw[black!30] (16.7,-7.75) -- (14.7,-4.75);
\draw[black!30] (17.7,-6.25) -- (16.7,-4.75);
\draw[black!30] (17.7,-6.25) -- (18.7,-4.75);
\draw[black!30] (23.7,-9.25) -- (23.7,-7.75);
\draw[black!30] (23.7,-7.75) -- (21.7,-6.25);
\draw[black!30] (23.7,-7.75) -- (25.7,-6.25);
\draw[black!30] (21.7,-6.25) -- (20.7,-4.75);
\draw[black!30] (21.7,-6.25) -- (22.7,-4.75);
\draw[black!30] (25.7,-6.25) -- (24.7,-4.75);
\draw[black!30] (25.7,-6.25) -- (26.7,-4.75);

\draw[thick,fill=white] (20,-14) rectangle (21.4,-13.5) {};
\node at (20.3,-13.75) {\tiny $n$};
\node at (21.1,-13.75) {\tiny $p$};
\draw[thick,fill=white] (26,-5) rectangle (27.4,-4.5) {};
\node at (26.3,-4.75) {\tiny $n$};
\node at (27.1,-4.75) {\tiny $p$};
\draw[thick,fill=white] (24,-5) rectangle (25.4,-4.5) {};
\node at (24.3,-4.75) {\tiny $n$};
\node at (25.1,-4.75) {\tiny $p$};
\draw[thick,fill=white] (22,-5) rectangle (23.4,-4.5) {};
\node at (22.3,-4.75) {\tiny $n$};
\node at (23.1,-4.75) {\tiny $p$};
\draw[thick,fill=white] (20,-5) rectangle (21.4,-4.5) {};
\node at (20.3,-4.75) {\tiny $n$};
\node at (21.1,-4.75) {\tiny $p$};
\draw[thick,fill=white] (18,-5) rectangle (19.4,-4.5) {};
\node at (18.3,-4.75) {\tiny $n$};
\node at (19.1,-4.75) {\tiny $p$};
\draw[thick,fill=white] (16,-5) rectangle (17.4,-4.5) {};
\node at (16.3,-4.75) {\tiny $n$};
\node at (17.1,-4.75) {\tiny $p$};
\draw[thick,fill=white] (14,-5) rectangle (15.4,-4.5) {};
\node at (14.3,-4.75) {\tiny $n$};
\node at (15.1,-4.75) {\tiny $p$};
\draw[thick,fill=white] (25,-6.5) rectangle (26.4,-6) {};
\node at (25.3,-6.25) {\tiny $n$};
\node at (26.1,-6.25) {\tiny $p$};
\draw[thick,fill=white] (21,-6.5) rectangle (22.4,-6) {};
\node at (21.3,-6.25) {\tiny $n$};
\node at (22.1,-6.25) {\tiny $p$};
\draw[thick,fill=white] (17,-6.5) rectangle (18.4,-6) {};
\node at (17.3,-6.25) {\tiny $n$};
\node at (18.1,-6.25) {\tiny $p$};
\draw[thick,fill=white] (16,-8) rectangle (17.4,-7.5) {};
\node at (16.3,-7.75) {\tiny $n$};
\node at (17.1,-7.75) {\tiny $p$};
\draw[thick,fill=white] (23,-8) rectangle (24.4,-7.5) {};
\node at (23.3,-7.75) {\tiny $n$};
\node at (24.1,-7.75) {\tiny $p$};
\draw[thick,fill=white] (23,-9.5) rectangle (24.4,-9) {};
\node at (23.3,-9.25) {\tiny $n$};
\node at (24.1,-9.25) {\tiny $p$};
\draw[thick,fill=white] (20,-11) rectangle (21.4,-10.5) {};
\node at (20.3,-10.75) {\tiny $n$};
\node at (21.1,-10.75) {\tiny $p$};
\draw[thick,fill=white] (20,-12.5) rectangle (21.4,-12) {};
\node at (20.3,-12.25) {\tiny $n$};
\node at (21.1,-12.25) {\tiny $p$};

\path[draw,thick,blue,use Hobby shortcut,closed=true]
(22.3,-5.5) .. (22.3,-5.7) .. (21.6,-6.25) .. (22,-6.7) .. (22.9,-7.75) .. (22.8,-9.25) .. (20.7,-10.75) .. (20.6,-13.25) .. (20.6,-13.75) .. (21.1,-14.2) .. (21.6,-13.75) .. (21.6,-13.25) .. (21.6,-10.75) .. (23.3,-9.7) .. (23.75,-9.25) .. (23.85,-8.5) .. (23.75,-7.75) .. (23.4,-7.3) .. (22.6,-6.4) .. (22.6,-6.1) .. (23.1,-5.7) .. (23.1,-5.4) .. (22.9,-5.1) .. (22,-4.3) .. (22,-5.1);

\node at (24,-11.5) {\Large $f_P^-$};

\draw[black!30] (20.7,-.75) -- (20.7,2.25);
\draw[black!30] (20.7,2.25) -- (23.7,3.75);
\draw[black!30] (20.7,2.25) -- (16.7,5.25);
\draw[black!30] (16.7,5.25) -- (17.7,6.75);
\draw[black!30] (16.7,5.25) -- (14.7,8.25);
\draw[black!30] (17.7,6.75) -- (16.7,8.25);
\draw[black!30] (17.7,6.75) -- (18.7,8.25);
\draw[black!30] (23.7,3.75) -- (23.7,5.25);
\draw[black!30] (23.7,5.25) -- (21.7,6.75);
\draw[black!30] (23.7,5.25) -- (25.7,6.75);
\draw[black!30] (21.7,6.75) -- (20.7,8.25);
\draw[black!30] (21.7,6.75) -- (22.7,8.25);
\draw[black!30] (25.7,6.75) -- (24.7,8.25);
\draw[black!30] (25.7,6.75) -- (26.7,8.25);

\draw[thick,fill=white] (20,-1) rectangle (21.4,-.5) {};
\node at (20.3,-.75) {\tiny $n$};
\node at (21.1,-.75) {\tiny $p$};
\draw[thick,fill=white] (26,8) rectangle (27.4,8.5) {};
\node at (26.3,8.25) {\tiny $n$};
\node at (27.1,8.25) {\tiny $p$};
\draw[thick,fill=white] (24,8) rectangle (25.4,8.5) {};
\node at (24.3,8.25) {\tiny $n$};
\node at (25.1,8.25) {\tiny $p$};
\draw[thick,fill=white] (22,8) rectangle (23.4,8.5) {};
\node at (22.3,8.25) {\tiny $n$};
\node at (23.1,8.25) {\tiny $p$};
\draw[thick,fill=white] (20,8) rectangle (21.4,8.5) {};
\node at (20.3,8.25) {\tiny $n$};
\node at (21.1,8.25) {\tiny $p$};
\draw[thick,fill=white] (18,8) rectangle (19.4,8.5) {};
\node at (18.3,8.25) {\tiny $n$};
\node at (19.1,8.25) {\tiny $p$};
\draw[thick,fill=white] (16,8) rectangle (17.4,8.5) {};
\node at (16.3,8.25) {\tiny $n$};
\node at (17.1,8.25) {\tiny $p$};
\draw[thick,fill=white] (14,8) rectangle (15.4,8.5) {};
\node at (14.3,8.25) {\tiny $n$};
\node at (15.1,8.25) {\tiny $p$};
\draw[thick,fill=white] (25,6.5) rectangle (26.4,7) {};
\node at (25.3,6.75) {\tiny $n$};
\node at (26.1,6.75) {\tiny $p$};
\draw[thick,fill=white] (21,6.5) rectangle (22.4,7) {};
\node at (21.3,6.75) {\tiny $n$};
\node at (22.1,6.75) {\tiny $p$};
\draw[thick,fill=white] (17,6.5) rectangle (18.4,7) {};
\node at (17.3,6.75) {\tiny $n$};
\node at (18.1,6.75) {\tiny $p$};
\draw[thick,fill=white] (16,5) rectangle (17.4,5.5) {};
\node at (16.3,5.25) {\tiny $n$};
\node at (17.1,5.25) {\tiny $p$};
\draw[thick,fill=white] (23,5) rectangle (24.4,5.5) {};
\node at (23.3,5.25) {\tiny $n$};
\node at (24.1,5.25) {\tiny $p$};
\draw[thick,fill=white] (23,3.5) rectangle (24.4,4) {};
\node at (23.3,3.75) {\tiny $n$};
\node at (24.1,3.75) {\tiny $p$};
\draw[thick,fill=white] (20,2) rectangle (21.4,2.5) {};
\node at (20.3,2.25) {\tiny $n$};
\node at (21.1,2.25) {\tiny $p$};
\draw[thick,fill=white] (20,.5) rectangle (21.4,1) {};
\node at (20.3,.75) {\tiny $n$};
\node at (21.1,.75) {\tiny $p$};

\path[draw,thick,blue,use Hobby shortcut,closed=true]
(22,7.3) .. (21.6,6.75) .. (22,6.3) .. (22.9,5.25) .. (22.8,3.75) .. (20.7,2.25) .. (20.6,-.25) .. (20.6,-.75) .. (21.1,-1.2) .. (21.6,-.75) .. (21.6,-.25) .. (21.6,2.25) .. (23.3,3.3) .. (23.75,3.75) .. (23.85,4.5) .. (23.75,5.25) .. (23.4,5.7) .. (22.6,6.6) .. (22.6,6.9) .. (22.9,7.3) .. (23.1,7.5) .. (23.4,7.9) .. (23.4,8.7) .. (22.5,7.9) .. (22.3,7.5);

\node at (24,1.5) {\Large $f_P^+$};

\end{tikzpicture}
\end{center}

By relabelling the vertices of $T_n$, we may consider
$$\F = \set{f_P^+,f_P^-}{P \text{ is a maximal chain in }T_n}$$
to be a subset of $\mathrm{Fn}(k,2)$, where $k = |V(T_n)|$. We claim that $\F$ is $\Ser(n+1)$-bounding.

For every non-maximal vertex $v$ of $T_n$, there is are maximal chains $P$ and $Q$ containing $v$, such that $P$ and $Q$ include different successors of $\s(v)$. But then $v \in \mathrm{dom}(f^+_P)$ and $v \in \mathrm{dom}(f^+_Q)$, and $f^+_P(v) \neq f^+_Q(v)$, so that $\{f^+_P(v), f^+_Q(v)\} = \{p,n\}$. If $v$ is a maximal vertex of $T_n$, then there is a unique path $P$ containing $v$, and we have $f^+_P(v) = p$ and $f^-_P(v) = n$.
Hence $\F$ is full. 

Suppose that $\G \sub \F$, that $D \sub \bigcup_{g \in \G}\mathrm{dom}(g)$, and that $|D| \geq n+1$. We claim that some two functions in $\G$ disagree on $D$. By a previous claim, $|D| > n$ implies that there are some $a,b,c \in D$ such that $b$ and $c$ are incomparable, and $\s(a) = b \wedge c$. Let $g_b,g_c \in \G$ such that $b \in \mathrm{dom}(g_b)$ and $c \in \mathrm{dom}(g_c)$. By our construction of the functions in $\F$, we must have $a \in \mathrm{dom}(g_b) \cap \mathrm{dom}(g_c)$, and $g_b(a) = g_b(\s(a)) \neq g_c(\s(a)) = g_c(a)$.

Hence $\F$ is $\Ser(n+1)$-bounding.
\end{proof}

It is obvious from the construction of the $T_n$ that $\card{V(T_n)} = k_n$ for every $n$, so this claim completes the proof of the theorem.
\end{proof}

\begin{theorem}\label{thm:lbSer}
Let $k_1, k_2, k_3, \dots$ denote the sequence defined in Lemma~\ref{lem:lbSer}. 
 $$\Ser(n) \geq k_{n-1}+1$$ for every $n > 1$. 
Moreover, $k_n > \frac{1}{2}(n-1) \log_2 (n-1) - \frac{1}{2}n + 2$, and thus
$$\textstyle \Ser(n) \,>\, \frac{1}{2}(n-1) \log_2 (n-1) - \frac{1}{2}n + 2$$ 
for every $n > 1$.
\end{theorem}
\begin{proof}
Let $f(x) = \frac{1}{2}(x \log_2 x - x + 1)$. Note that the second derive of $f(x)$ is positive everywhere on the interval $(0,\infty)$: 
$$\textstyle f'(x) = \frac{1}{2 \ln 2}(1 + \ln x) - \frac{1}{2} \qquad \text{ and } \qquad f''(x) \,=\, \frac{1}{(2 \ln 2) x} \,>\, 0.$$
This implies that $f(x)$ is \emph{convex} on $(0,\infty)$, meaning that if $0 < a < b$, the graph of $y = f(x)$ contains no points strictly above the line segment connecting $(a,f(a))$ and $(b,f(b))$.
In particular, if $x \in (0,\infty)$ then
$$\textstyle f(x+\frac{1}{2}) \,\leq\, \frac{1}{2}(f(x)+f(x+1)),$$
and this implies $\textstyle f \!\left( \left \lfloor \frac{n}{2} \right \rfloor \right) + f \!\left( \left \lfloor \frac{n+1}{2} \right \rfloor \right) \,\geq\, 2 f \!\left( \frac{n}{2} \right)$ for every $n \in \N$.

Let $k_1, k_2, k_3, \dots$ denote the sequence defined in Lemma~\ref{lem:lbSer}. We prove by induction on $n$ that $f(n) < k_n$ for all $n \in \N$. The base case is straightforward: $f(1) = 0 < 1 = k_1$. For the inductive step, fix $n > 1$ and suppose $k_\ell < f(\ell)$ whenever $\ell < n$. Using the conclusion of the previous paragraph and the inductive hypothesis, we have
\begin{align*}
k_n &\,=\, \left \lfloor \frac{n}{2} \right \rfloor + k_{\left \lfloor \frac{n}{2} \right \rfloor} + k_{\left \lfloor \frac{n+1}{2} \right \rfloor} \,>\, \left \lfloor \frac{n}{2} \right \rfloor + \textstyle f \!\left( \left \lfloor \frac{n}{2} \right \rfloor \right) + f \!\left( \left \lfloor \frac{n+1}{2} \right \rfloor \right) \,\geq\, \textstyle \frac{n-1}{2} + 2 f \!\left( \frac{n}{2} \right) \vphantom{f^{f^{f^f}}} \\
&\!\!= \textstyle \frac{n-1}{2} + \frac{n}{2}\log_2 \frac{n}{2} - \frac{n}{2} + 1 = \textstyle \frac{1}{2}(n\log_2 \frac{n}{2}) + \frac{1}{2} = \textstyle \frac{1}{2}(n \log_2 n - n + 1) = f(n). \vphantom{f^{f^{f^f}}f_{f_{f_f}}}
\end{align*}
By induction, $f(n) < k_n$ for all $n \in \N$ as claimed.

By Lemma~\ref{lem:lbSer}, there is for every $n \in \N$ a $\Ser(n+1)$-bounding full subset of $\mathrm{Fn}(k_n,2)$.
Thus for every $n > 1$,
$$\textstyle \Ser(n) \,\geq\, k_{n-1}+1 \,>\, f(n-1)+1 \,=\, \frac{1}{2}(n-1) \log_2 (n-1) - \frac{1}{2}n + 2,$$
as claimed.
\end{proof}

\begin{corollary}
Let $k_1, k_2, k_3, \dots$ denote the sequence defined in Lemma~\ref{lem:lbSer}. Then for every $n$, $\Hyp(n) \geq k_n$ and consequently,
$$\textstyle \Hyp(n) \,>\, \frac{1}{2} n \log_2 n - \frac{1}{2}n + \frac{1}{2}.$$
\end{corollary}
\begin{proof}
The first inequality follows from the previous theorem and Theorem~\ref{thm:Ser}, which asserts that
$\Hyp(n) \geq \Ser(n+1)-1$ for all $n$. The second inequality is then a direct consequence of the lower bound on the $k_n$ found in the proof of Theorem~\ref{thm:lbSer}.
\end{proof}

The bound in this corollary implies that there is, for every $n$, a hypergraph on $k_n$ vertices containing no partitions of size greater than $n$. Indeed, the proof of Lemma~\ref{lem:lbSer} gives us such hypergraphs: for each $n$, if $\B_n$ denotes the set of all branches through the tree $T_n$, then the hypergraph $(T_n,\B_n)$ is a witness to the assertion $\Hyp(n) \geq |T_n| = k_n$.

The next corollary simply converts the bound from Theorem~\ref{thm:lbSer} into a form that reveals its asymptotic growth rate.

\begin{corollary}
For every $n \in \N$,
$$\textstyle \Ser(n) \,>\, \frac{1}{2}n \log_2 n - \frac{1}{2}n - \frac{1}{2}\log_2 n + \frac{\ln 16 - 1}{\ln 4}.$$
\end{corollary}
\begin{proof}
First note that if $0 < x < 1$, then, using the Taylor series for $\ln(1-x)$,
\begin{align*}
\textstyle \left( \frac{1}{x} - 1 \right) \ln(1-x) & \,=\, \left( \frac{1}{x} - 1 \right) \sum_{k=1}^\infty\frac{-x^k}{k} \,=\, \sum_{k=1}^\infty\frac{-x^k}{k} \left( \frac{1}{x} - 1 \right) \\
& \,=\, x-1+\frac{x^2}{2} - \frac{x}{2} + \frac{x^3}{3} - \frac{x^2}{3} + \frac{x^4}{4} - \frac{x^3}{4} + \dots \\
&\,=\, -1 + \sum_{k=1}^\infty \frac{x^k}{k(k+1)} \,\geq\, -1.
\end{align*}

If $n > 1$, then putting $x = \frac{1}{n}$ shows
$$\textstyle (n-1) \ln \! \left( 1-\frac{1}{n} \right) \,=\, (n-1) \ln \! \left( \frac{n-1}{n} \right) \,=\, (n-1) \ln ( n-1 ) - (n-1)\ln ( n ) \,\geq\, -1$$
or, after rearranging and dividing by $\ln 2$,
$$\textstyle (n-1) \log_2 (n-1) \,\geq\, n \log_2 n - \log_2 n - \frac{1}{\ln 2}.$$
Combining this with Theorem~\ref{thm:lbSer}, we get
\begin{align*}
\textstyle \Ser(n) & \,>\, \textstyle \frac{1}{2}(n-1) \log_2 (n-1) - \frac{1}{2}n + 2 \\
& \textstyle \,\geq\, \frac{1}{2} \left( n \log_2 n - \log_2 n - \frac{1}{\ln 2} \right) - \frac{1}{2}n + 2 \\
& \,=\, \textstyle \frac{1}{2}n \log_2 n - \frac{1}{2}n - \frac{1}{2}\log_2 n + \frac{\ln 16 - 1}{\ln 4}
\end{align*}
for all $n > 1$. It is easy to check that this bound also holds for $\Ser(1)=1$.
\end{proof}

Asymptotically, our lower and upper bounds for $\Ser(n)$ and $\Hyp(n)$ differ only by a constant multiple, so we have proved that
$\Ser(n) = \Theta(n \log n)$ and $\Hyp(n) = \Theta(n \log n)$. We conjecture that $\Hyp(n) = n \ln n + o(n \log n)$, so that it is the upper bound, rather than the lower bound, that gives the correct coefficient for the $n \log n$ term. 

For small values of $n$, our upper and lower bounds for $\Ser(n)$ and $\Hyp(n)$ agree, and therefore give us the exact values of these numbers. We record this observation in the next theorem, and compute as well a few other values of $\Ser(n)$ and $\Hyp(n)$.

\begin{theorem}\label{thm:six}
The first six values of $\Ser(n)$ and $\Hyp(n)$ are:
\begin{multicols}{2}
\begin{itemize}
\item[$\circ$] $\Ser(1) = 1$
\item[$\circ$] $\Ser(2) = 2$
\item[$\circ$] $\Ser(3) = 4$
\item[$\circ$] $\Ser(4) = 6$
\item[$\circ$] $\Ser(5) = 9$
\item[$\circ$] $\Ser(6) = 11$

\item[$\circ$] $\Hyp(1) = 1$
\item[$\circ$] $\Hyp(2) = 3$
\item[$\circ$] $\Hyp(3) = 5$
\item[$\circ$] $\Hyp(4) = 8$
\item[$\circ$] $\Hyp(5) = 10$
\item[$\circ$] $\Hyp(6) = 14$
\end{itemize}
\end{multicols}
\end{theorem}
\begin{proof}
One may easily check that the sequence $k_1,k_2,k_3,\dots$ defined in Lemma~\ref{lem:lbSer} begins $1,3,5,8,10,13,\dots$.
By Theorems~\ref{thm:Ser}, \ref{thm:upper}, and \ref{thm:lbSer}, we know that 
$$k_n+1 \,\leq\, \Ser(n+1) \,\leq\, \Hyp(n)+1$$
for all $n > 1$ and
$$k_n \,\leq\, \Hyp(n) \,\leq\, \sum_{k=1}^{n}\frac{k}{n}$$
for all $n$. 
For $n = 1,2,3,4$ these bounds match (after rounding the upper bound for $\Hyp(n)$ down to the nearest integer), thus proving that 
\begin{multicols}{2}
\begin{itemize}
\item[$\circ$] $\Ser(2) = 2$
\item[$\circ$] $\Ser(3) = 4$
\item[$\circ$] $\Ser(4) = 6$
\item[$\circ$] $\Ser(5) = 9$

\item[$\circ$] $\Hyp(1) = 1$
\item[$\circ$] $\Hyp(2) = 3$
\item[$\circ$] $\Hyp(3) = 5$
\item[$\circ$] $\Hyp(4) = 8$
\end{itemize}
\end{multicols}
\noindent Furthermore, it is easy to see that $\Ser(1) = 1$. Putting $n=5$ into the inequalities above, we get
$$11 \leq \Ser(6) \leq 12 \ \text { and } \ 10 \leq \Hyp(5) \leq 11,$$
and for $n=6$ we get
$$13 \leq \Hyp(6) \leq 14.$$
To finish the proof of the theorem, we must show that $\Hyp(5) \leq 10$ (which implies $\Ser(6) \leq 11$) and that $\Hyp(6) \geq 14$.

To show $\Hyp(6) \geq 14$, it suffices to exhibit a hypergraph on $14$ vertices containing no partitions of size greater than $6$. Here are two such examples:

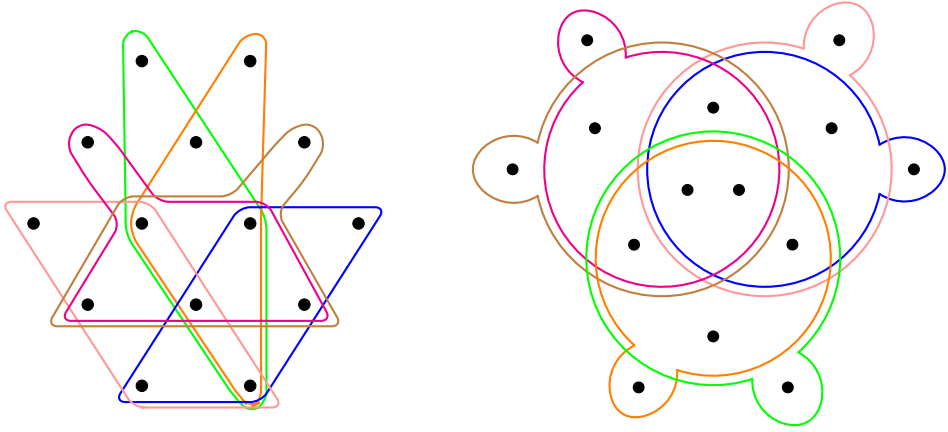
\vspace{2mm}
\begin{center}
\begin{tikzpicture}[xscale=.72,yscale=.72]

\draw[fill=black] (2,0) circle (3pt);
\draw[fill=black] (4,0) circle (3pt);
\draw[fill=black] (1,1.5) circle (3pt);
\draw[fill=black] (3,1.5) circle (3pt);
\draw[fill=black] (5,1.5) circle (3pt);
\draw[fill=black] (0,3) circle (3pt);
\draw[fill=black] (2,3) circle (3pt);
\draw[fill=black] (4,3) circle (3pt);
\draw[fill=black] (6,3) circle (3pt);
\draw[fill=black] (1,4.5) circle (3pt);
\draw[fill=black] (3,4.5) circle (3pt);
\draw[fill=black] (5,4.5) circle (3pt);
\draw[fill=black] (2,6) circle (3pt);
\draw[fill=black] (4,6) circle (3pt);

\draw[rounded corners,green,thick] (4.3,2) -- (4.3,-.3) -- (4,-.5) -- (3.7,-.2) -- (1.7,2.8) -- (1.7,3.2) -- (1.65,6.5) -- (2,6.6) -- (4.3,3.1) -- (4.3,2);
\draw[rounded corners,orange,thick] (4.2,2) -- (4.2,-.3) -- (4,-.4) -- (3.8,-.2) -- (1.8,2.8) -- (1.8,3.2) -- (3.9,6.5) -- (4.3,6.5) -- (4.2,3.1) -- (4.2,2);
\draw[rounded corners,blue,thick] (2,-.3) -- (4.2,-.3) -- (6.5,3.3) -- (3.8,3.3) -- (1.5,-.3) -- (2,-.3);
\draw[rounded corners,red!40,thick] (2,-.4) -- (1.85,-.4) -- (-.6,3.4) -- (2.15,3.4) -- (4.6,-.4) -- (2,-.4);
\draw[rounded corners,magenta,thick] (1,1.2) -- (5.5,1.2) -- (4.3,3.4) -- (2.3,3.4) -- (1.5,4.5) -- (1.2,4.8) -- (.8,4.85) -- (.6,4.45) -- (1,3.8) -- (1.6,3) -- (.5,1.2) -- (1,1.2);
\draw[rounded corners,brown,thick] (5,1.1) -- (.25,1.1) -- (1.65,3.5) -- (3.65,3.5) -- (4.5,4.5) -- (4.8,4.8) -- (5.2,4.85) -- (5.4,4.45) -- (5,3.8) -- (4.5,3.2) -- (5.7,1.1) -- (5,1.1);

\begin{scope}[shift={(13.5,4)},xscale=.95,yscale=.95]

\draw[fill=black] (-1.5,-.4) circle (3pt);
\draw[fill=black] (-.5,-.4) circle (3pt);
\draw[fill=black] (1.3,.8) circle (3pt);
\draw[fill=black] (-3.3,.8) circle (3pt);
\draw[fill=black] (-1,-3.25) circle (3pt);

\draw[thick,blue] (0,0) circle (65pt);
\draw[white,ultra thick] (2.25,-.47) -- (2.25,.47);
\draw[white,ultra thick] (2.3,-.47) -- (2.3,.47);
\path[draw,blue,thick,use Hobby shortcut,closed=false]
(2.22,.47) .. (3.3,.4) .. (3.5,0) .. (3.3,-.4) .. (2.22,-.47);

\draw[fill=black] (-1,1.2) circle (3pt);
\draw[fill=black] (2.9,0) circle (3pt);
\draw[fill=black] (1.45,2.511) circle (3pt);

\begin{scope}[rotate={60},xscale=1.08,yscale=1.08]
\draw[thick,red!40] (0,0) circle (65pt);
\draw[white,ultra thick] (2.25,-.47) -- (2.25,.47);
\draw[white,ultra thick] (2.3,-.47) -- (2.3,.47);
\path[draw,red!40,thick,use Hobby shortcut,closed=false]
(2.22,.47) .. (3.2,.4) .. (3.4,0) .. (3.2,-.4) .. (2.22,-.47);
\end{scope}

\begin{scope}[shift={(-2,0)}]
\begin{scope}[rotate={120}]

\draw[thick,magenta] (0,0) circle (65pt);
\draw[white,ultra thick] (2.25,-.47) -- (2.25,.47);
\draw[white,ultra thick] (2.3,-.47) -- (2.3,.47);
\path[draw,magenta,thick,use Hobby shortcut,closed=false]
(2.22,.47) .. (3.3,.4) .. (3.5,0) .. (3.3,-.4) .. (2.22,-.47);

\draw[fill=black] (-1,1.2) circle (3pt);
\draw[fill=black] (2.9,0) circle (3pt);
\draw[fill=black] (1.45,2.511) circle (3pt);

\begin{scope}[rotate={60},xscale=1.08,yscale=1.08]
\draw[thick,brown] (0,0) circle (65pt);
\draw[white,ultra thick] (2.25,-.47) -- (2.25,.47);
\draw[white,ultra thick] (2.3,-.47) -- (2.3,.47);
\path[draw,brown,thick,use Hobby shortcut,closed=false]
(2.22,.47) .. (3.2,.4) .. (3.4,0) .. (3.2,-.4) .. (2.22,-.47);
\end{scope}

\end{scope}
\end{scope}

\begin{scope}[shift={(-1,-1.73)}]
\begin{scope}[rotate={240}]

\draw[thick,orange] (0,0) circle (65pt);
\draw[white,ultra thick] (2.25,-.47) -- (2.25,.47);
\draw[white,ultra thick] (2.3,-.47) -- (2.3,.47);
\path[draw,orange,thick,use Hobby shortcut,closed=false]
(2.22,.47) .. (3.3,.4) .. (3.5,0) .. (3.3,-.4) .. (2.22,-.47);

\draw[fill=black] (-1,1.2) circle (3pt);
\draw[fill=black] (2.9,0) circle (3pt);
\draw[fill=black] (1.45,2.511) circle (3pt);

\begin{scope}[rotate={60},xscale=1.08,yscale=1.08]
\draw[thick,green] (0,0) circle (65pt);
\draw[white,ultra thick] (2.25,-.47) -- (2.25,.47);
\draw[white,ultra thick] (2.3,-.47) -- (2.3,.47);
\path[draw,green,thick,use Hobby shortcut,closed=false]
(2.22,.47) .. (3.2,.4) .. (3.4,0) .. (3.2,-.4) .. (2.22,-.47);
\end{scope}

\end{scope}
\end{scope}

\end{scope}

\end{tikzpicture}
\end{center}
\vspace{2mm}

\noindent For each of these hypergraphs, it is not obvious that there are no partitions of size $>\!6$, but it can be checked manually with sufficient patience. For each subset of the hyperedges, one merely has to check that no more than $6$ vertices are contained in exactly one member of the subset. This is trivial for each particular subset of the hyperedges; the only trouble is that there are $2^6 = 64$ cases to check for each hypergraph.

To show $\Hyp(5) \leq 10$, let us suppose (aiming for a contradiction) that $\Hyp(5) \geq 11$. By Lemma~\ref{lem:econ2}, if $\Hyp(5) \geq 11$ then there is an economical hypergraph $(V,\HH)$ with $|V| \geq 11$ containing no partitions of size greater than $5$. 

As in the proof of Theorem~\ref{thm:quadbound}, we may for every $E \in \HH$ find some $v_E \in V$ such that $v_E$ is isolated in $(V,\HH \setminus \{E\})$. Then $\set{v_E}{E \in \HH}$ and $\HH$ form a partition in $(V,\HH)$ of size $|\HH|$. Hence $|\HH| \leq 5$.
Also, every $E \in \HH$ has size $\leq\! 5$, because $E$ and $\{E\}$ form a partition in $(V,\HH)$ of size $|E|$. 

If $|E| \leq 3$ for every $E \in \HH$, then (as there are no isolated points) $|\HH| \leq 5$ and $|V| \geq 11$ imply there are at least $7$ vertices each contained in exactly one member of $\HH$. But then these $7$ or more vertices together with $\HH$ form a partition of size $\geq\! 7$. Thus $\HH$ contains hyperedges of size $4$ or $5$.

In fact, $\HH$ must contain hyperedges of size $5$. To see this, suppose (aiming for a contradiction) that $|E| \leq 4$ for every $E \in \HH$. Using the previous paragraph, fix $E \in \HH$ with $|E| = 4$. Observe that $|F \setminus E| \leq 2$ for all $F \in \HH$: otherwise, we would have both $|F \setminus E| \geq 3$ and (using $|F| \leq 4$, which together with $|F \setminus E| \geq 3$ implies $|E \cap F| \leq 1$) also $|E \setminus F| \geq 3$, in which case taking $X = E \triangle F$ and $\P = \{E,F\}$ would give a partition of size $\geq \! 6$.
Let $W = V \setminus E$. Then $|W| \geq 7$ and (we just showed that) each member of $\HH \setminus \{E\}$ contains at most two points of $W$. As $|\HH \setminus \{E\}| \leq 4$, this implies there are at least $6$ points of $W$ each of which is contained in exactly one member of $\HH \setminus \{E\}$. But then taking $X$ to be this set of $\geq\! 6$ points and $\P = \HH \setminus \{E\}$ gives a partition of size $\geq\! 6$.

Thus $\HH$ contains a hyperedge $E$ of size $5$. As in the previous paragraph, observe that $|F \setminus E| \leq 2$ for all $F \in \HH$: otherwise, we would have both $|F \setminus E| \geq 3$ and $|E \setminus F| \geq 3$, in which case $E \triangle F$ and $\{E,F\}$ would form a partition of size $\geq\! 6$. 

Let $W = V \setminus E$ and let $\G = \set{F \cap W}{F \in \HH \setminus \{E\}}$. Note that $|\G| \leq 4$ and that the hypergraph $(W,\G)$ has no isolated points. Also $|W| = |V| - |E| \geq 6$, and by the previous paragraph, each member of $\G$ has size $\leq\! 2$. $\G$ cannot consist of singletons because $|\G| < |Y|$; thus there is some $G \in \G$ with $|G| = 2$. 

Similarly, let $X = W \setminus G$ and let $\F = \set{F \setminus G}{F \in \G \setminus \{G\}}$. Then, arguing as in the previous paragraph, $|X| \geq 4$ and $|\F| \leq 3$, so $\F$ cannot consist of singletons. Thus there is some $F \in \F$ such that $|F| = 2$. But since the members of $\G$ all have size $\leq\! 2$, this implies $F \in \G$, and that $G \cap F = \0$.

Thus there are $G,F \in \G$ with $|G| = |F| = 2$ and $G \cap F = \0$.
By the definition of $\G$, this means that there are $A,B \in \HH$ such that $|A \setminus E| = |B \setminus E| = 2$ and $(A \setminus E) \cap (B \setminus E) = \0$. 
 
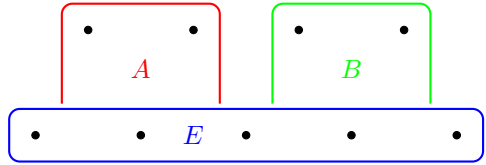
\vspace{1mm}
\begin{center}
\begin{tikzpicture}[xscale=.35,yscale=.35]

\draw[fill=black] (0,0) circle (4pt);
\draw[fill=black] (4,0) circle (4pt);
\draw[fill=black] (8,0) circle (4pt);
\draw[fill=black] (12,0) circle (4pt);
\draw[fill=black] (16,0) circle (4pt);
\draw[fill=black] (2,4) circle (4pt);
\draw[fill=black] (6,4) circle (4pt);
\draw[fill=black] (10,4) circle (4pt);
\draw[fill=black] (14,4) circle (4pt);

\draw[rounded corners,blue,thick] (0,-1) -- (17,-1) -- (17,1) -- (-1,1) -- (-1,-1) -- (0,-1);
\draw[rounded corners,red,thick] (1,1.2) -- (1,5) -- (7,5) -- (7,1.2);
\draw[rounded corners,green,thick] (9,1.2) -- (9,5) -- (15,5) -- (15,1.2);
\node[blue] at (6,0) {\small $E$};
\node[red] at (4,2.5) {\small $A$};
\node[green] at (12,2.5) {\small $B$};

\end{tikzpicture}
\end{center}
\vspace{1mm}

If $|E \setminus (A \cup B)| \geq 2$, then taking $\P = \{A,B,E\}$ and
$$X = (E \setminus (A \cup B)) \cup (A \setminus E) \cup (B \setminus E)$$
gives a partition of size $\geq\! 6$. Thus $|E \setminus (A \cup B)| \leq 1$. But recall that $|A|, |B| \leq 5$, so this implies that
$$A \setminus (E \cap B) \neq \0 \qquad \text{ and } \qquad B \setminus (E \cap A) \neq \0,$$
which means 
$A \triangle B$
has size $\geq\! 6$. But then $A \triangle B$ and $\{A,B\}$ form a partition of size $\geq\! 6$.
Hence $\Hyp(5) \leq 10$.
\end{proof}

One question left open by the previous proof is whether $\Ser(7)$ is equal to $14$ or to $15$. (Using the bounds stated in the proof, it must be one of these.)
Either answer would be interesting. If $\Ser(7) = 14$, then we would know that the upper bound $\Ser(n) \leq \Hyp(n-1)+1$ can be strict. If $\Ser(7) = 15$, then we would know that the lower bound $\Ser(n) \geq k_{n-1}+1$ can be strict. At the moment we do not know that either of these inequalities can be strict, because there are no cases where we can compute $\Ser(n)$, except where $n=1$ or where (as in the previous proof) our upper and lower bounds match.

We suspect that both inequalities can be strict. Concerning the question of whether the bound $\Ser(n) \geq k_{n-1}+1$ can be strict, let us remark that even if the families of functions constructed in the proof of Lemma~\ref{lem:lbSer} are optimal in size, they are not unique. For example, the picture following Definition~\ref{def:full} gives an example of a $\Ser(3)$-bounding family $\F$ that is different from the one constructed from $T_2$ in the proof of Lemma~\ref{lem:lbSer}. Similarly, the following picture shows a $\Ser(4)$-bounding family that seems completely unrelated to the one constructed from $T_3$:

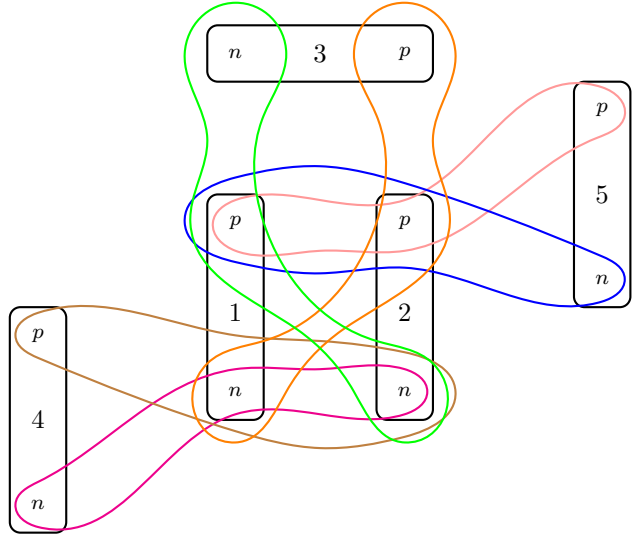
\vspace{5mm}
\begin{center}
\begin{tikzpicture}[xscale=1.5,yscale=1.5]

\draw[rounded corners,thick] (0,0) rectangle (.5,2) {};
\node at (.25,.25) {\scriptsize $n$};
\node at (.25,1.75) {\scriptsize $p$};
\draw[rounded corners,thick] (1.5,0) rectangle (2,2) {};
\node at (1.75,.25) {\scriptsize $n$};
\node at (1.75,1.75) {\scriptsize $p$};
\draw[rounded corners,thick] (-1.75,-1) rectangle (-1.25,1) {};
\node at (-1.5,-.75) {\scriptsize $n$};
\node at (-1.5,.75) {\scriptsize $p$};
\draw[rounded corners,thick] (3.25,1) rectangle (3.75,3) {};
\node at (3.5,1.25) {\scriptsize $n$};
\node at (3.5,2.75) {\scriptsize $p$};
\draw[rounded corners,thick] (0,3) rectangle (2,3.5) {};
\node at (.25,3.25) {\scriptsize $n$};
\node at (1.75,3.25) {\scriptsize $p$};

\node at (.25,.95) {\small $1$};
\node at (1.75,.95) {\small $2$};
\node at (1,3.25) {\small $3$};
\node at (-1.5,0) {\small $4$};
\node at (3.5,2) {\small $5$};

\path[draw,red!40,thick,use Hobby shortcut,closed=true]
(.05,1.75) .. (.25,1.95) .. (1,1.95) .. (1.75,1.95) .. (3.5,2.95) .. (3.7,2.7) .. (3.45,2.5) .. (1.75,1.5) .. (1,1.5) .. (.25,1.5);

\path[draw,magenta,thick,use Hobby shortcut,closed=true]
(1.95,.25) .. (1.75,.05) .. (1,.05) .. (.25,.05) .. (-1.5,-.95) .. (-1.7,-.75) .. (-1.5,-.55) .. (.25,.45) .. (1,.45) .. (1.75,.45);

\path[draw,brown,thick,use Hobby shortcut,closed=true]
(2.2,.25) .. (1.75,-.15) .. (1,-.25) .. (.25,-.1) .. (-1.5,.55) .. (-1.7,.75) .. (-1.5,.95) .. (.25,.75) .. (1,.7) .. (1.75,.6);

\path[draw,blue,thick,use Hobby shortcut,closed=true]
(-.2,1.75) .. (.25,2.15) .. (1,2.25) .. (1.75,2.1) .. (3.5,1.45) .. (3.7,1.25) .. (3.5,1.05) .. (1.75,1.35) .. (1,1.3) .. (.25,1.4);

\path[draw,orange,thick,use Hobby shortcut,closed=true]
(.25,-.2) .. (.7,.25) .. (1.2,.8) .. (2.15,1.75) .. (2,2.5) .. (2.2,3.25) .. (1.75,3.7) .. (1.3,3.25) .. (1.5,2.7) .. (1.5,1.75) .. (.5,.7) .. (0,.5);

\path[draw,green,thick,use Hobby shortcut,closed=true]
(1.75,-.2) .. (1.3,.25) .. (.8,.8) .. (-.15,1.75) .. (0,2.5) .. (-.2,3.25) .. (.25,3.7) .. (.7,3.25) .. (.5,2.7) .. (.5,1.75) .. (1.5,.7) .. (2,.5);

\end{tikzpicture}
\end{center}
\vspace{5mm}

\section{Conditionally convergent series}\label{sec:series}

In this section we apply the results of Sections~\ref{sec:reduction} and \ref{sec:combinatorics1} to prove the second theorem stated in the introduction.

In what follows, $\bar a$ always denotes a sequence $\seq{a_n}{n \in \N}$ of real numbers; similarly, $\bar a^i$ always denotes a sequence $\seq{a^i_n}{n \in \N}$.
If $\bar a$ is a sequence of real numbers and $A \sub \N$, then $\sum(\bar a,A)$ denotes the subseries $\sum_{n \in A}a_n$.

\begin{definition}\label{def:ser}
For each $n \in \N$, let $\ser(n)$ denote the least $k \in \N$ with the following property:
\begin{itemize}
\item[$(\ddagger)_n$] For any $k$ conditionally convergent series, there is some $A \sub \N$ sending at least $n$ of those series to infinity.
\end{itemize}
If there is no such $k \in \N$, then we say that $\tilde \Ser(n)$ is not well-defined.
\end{definition}

\begin{theorem}\label{thm:main}
The number $\ser(n)$ is well-defined for each $n \in \N$. Furthermore, $\ser(n) \leq \Ser(n)$ for all $n$.
\end{theorem}
\begin{proof}
Let $k \geq \Ser(n)$ and let $\{\bar a^1, \bar a^2, \dots, \bar a^k\}$ be a collection of $k$ conditionally convergent series. For each $\ell \in \{1,2,\dots,k\}$, define the following two ideals:
$$\textstyle \I_\ell = \set{A \sub \N}{\sum(\bar a^i,\set{n \in A}{a_n > 0}) \text{ converges}},$$
$$\textstyle \J_\ell = \set{A \sub \N}{\sum(\bar a^i,\set{n \in A}{a_n < 0}) \text{ converges}}.$$
It is not difficult to see that $\I_\ell$ and $\J_\ell$ are incompatible ideals on $\N$ for every $\ell \leq k$. 
It also is not difficult to see that if $A \in \I_\ell \setminus \J_\ell$, then $\sum(\bar a^\ell,A) = -\infty$, and if $A \in \J_\ell \setminus \I_\ell$, then $\sum(\bar a^\ell,A) = \infty$. Hence if $A$ chooses between $\I_\ell$ and $\J_\ell$, then $A$ sends $\bar a^\ell$ to infinity. As $k \geq \Ser(n)$, there is some $A \sub \N$ that chooses between $\I_\ell$ and $\J_\ell$ for at least $n$ distinct values of $\ell$; hence there is some $A \sub \N$ that sends at least $n$ of these $k$ series to infinity.
\end{proof}

The inequality stated in the theorem above can be strict:
It is proved in \cite{Brian} is that $\ser(3) = 3$, but we saw in Theorem~\ref{thm:six} that $\Ser(3) = 4$.

The inequality $\ser(n) \leq \Ser(n)$ gives us an $n \log n$-type upper bound for $\ser(n)$ via the results of Section~\ref{sec:combinatorics1}, namely $\ser(n) \leq n \ln n + \g n - \ln n +\frac{3}{2} - \g$. Our next theorem provides a nontrivial lower bound for $\ser(n)$.

\begin{theorem}\label{thm:lbser}
$\ser(n) \geq 2n-5$ for every $n \in \N$.
\end{theorem}
\begin{proof}
To prove this theorem, we construct for every $n \in \N$ a collection of $2n$ conditionally convergent series, such that no more than $n+2$ of the series can be sent to infinity simultaneously. This shows that $\ser(n+3) > 2n$, or equivalently that $\ser(n) \geq 2n-5$, for every $n$. The proof expands on an idea of Nazarov \cite{Nazarov} presented in \cite[Section 3]{Brian}, where the bound $\ser(3) \geq 4$ is proved. The presentation here follows \cite[Section 3]{Brian} as closely as possible.

Fix $n \in \N$. Partition $\N$ into adjacent intervals $I_1, I_2, I_3, \dots$ (the lengths of which will be specified later in the proof). For each $m \geq 0$, let $b_m$ denote the length of the interval $I_m$.
The function $m \mapsto b_m$ is rapidly increasing (and just how rapidly it should increase is specified below). For convenience, we shall take each $b_m$ to be an even number, so that the first member of every interval $I_m$ is an odd number.

We now define our collection of $2n$ series (modulo the as-yet-undefined sequence of $b_m$'s) by specifying the terms of each one on each of the intervals $I_m$. Let us denote the series by $\bar a^1, \bar a^2, \dots, \bar a^{2n}$, and write $\bar a^i = \seq{a^i_k}{k \in \N}$. Given $m = \ell n+j$, $0 \leq j < n$, define $a^i_k$ on the interval $I_m$ as follows:
\begin{itemize}
\item[$\circ$] If $i \leq n$ and $i \neq j$, then $a_k^i = \frac{1}{m}$ for odd $k$ and $a_k^i = -\frac{1}{m}$ for even $k$.
\item[$\circ$] $a_k^j = -\frac{1}{m}$ for odd $m$ and $a_k^j = \frac{1}{m}$ for even $m$.
\item[$\circ$] If $i \geq n$ and $i \neq n+j$, then $a_k^i = 0$ for all $k$.
\item[$\circ$] $a_k^{n+j} = -\frac{1}{b_m}$ for odd $k$ and $a_k^{n+j} = \frac{1}{b_m}$ for even $k$.
\end{itemize}
Explicitly, our $2n$ series look like this on $I_m$ when $m \equiv j$ (modulo $n$):
\renewcommand*{\arraystretch}{2}
$$\begin{matrix}

\text{ series } 1: \ \, & \textstyle + \frac{1}{m} \ & \textstyle - \frac{1}{m} \ & \textstyle + \frac{1}{m} \ & \textstyle - \frac{1}{m} \ & \textstyle + \frac{1}{m} \ & \textstyle - \frac{1}{m} \ & \textstyle + \frac{1}{m} \ & \textstyle - \frac{1}{m} & \ \dots \\

\text{ series } 2: \ \, & \textstyle + \frac{1}{m} \ & \textstyle - \frac{1}{m} \ & \textstyle + \frac{1}{m} \ & \textstyle - \frac{1}{m} \ & \textstyle + \frac{1}{m} \ & \textstyle - \frac{1}{m} \ & \textstyle + \frac{1}{m} \ & \textstyle - \frac{1}{m} & \ \dots  \\

\vdots & \vdots & \vdots & \vdots & \vdots & \vdots & \vdots & \vdots & \vdots \\

\text{ series } j: \ \, & \textstyle - \frac{1}{m} \ & \textstyle + \frac{1}{m} \ & \textstyle - \frac{1}{m} \ & \textstyle + \frac{1}{m} \ & \textstyle - \frac{1}{m} \ & \textstyle + \frac{1}{m} \ & \textstyle - \frac{1}{m} \ & \textstyle + \frac{1}{m} & \ \dots  \\

\vdots & \vdots & \vdots & \vdots & \vdots & \vdots & \vdots & \vdots & \vdots \\

\text{ series } n: \ \, & \textstyle + \frac{1}{m} \ & \textstyle - \frac{1}{m} \ & \textstyle + \frac{1}{m} \ & \textstyle - \frac{1}{m} \ & \textstyle + \frac{1}{m} \ & \textstyle - \frac{1}{m} \ & \textstyle + \frac{1}{m} \ & \textstyle - \frac{1}{m} & \ \dots  \\

\text{ series } n\!+\!1: \ \, & \textstyle + 0 \ & \textstyle + 0 \ & \textstyle + 0 \ & \textstyle + 0 \ & \textstyle + 0 \ & \textstyle + 0 \ & \textstyle + 0 \ & \textstyle + 0 & \ \dots  \\

\text{ series } n\!+\!2: \ \, & \textstyle + 0 \ & \textstyle + 0 \ & \textstyle + 0 \ & \textstyle + 0 \ & \textstyle + 0 \ & \textstyle + 0 \ & \textstyle + 0 \ & \textstyle + 0 & \ \dots  \\

\vdots & \vdots & \vdots & \vdots & \vdots & \vdots & \vdots & \vdots & \vdots \\

\text{ series } n\!+\!j: \ \, & \textstyle + \frac{1}{b_m} \, & \textstyle - \frac{1}{b_m} \, & \textstyle + \frac{1}{b_m} \, & \textstyle - \frac{1}{b_m} \, & \textstyle + \frac{1}{b_m} \, & \textstyle - \frac{1}{b_m} \, & \textstyle + \frac{1}{b_m} \, & \textstyle - \frac{1}{b_m} & \ \dots  \\

\vdots & \vdots & \vdots & \vdots & \vdots & \vdots & \vdots & \vdots & \vdots \\

\text{ series } 2n: \ \, & \textstyle + 0 \ & \textstyle + 0 \ & \textstyle + 0 \ & \textstyle + 0 \ & \textstyle + 0 \ & \textstyle + 0 \ & \textstyle + 0 \ & \textstyle + 0 & \ \dots 

\end{matrix}$$
Assuming $\lim_{m \to \infty}b_m = \infty$, it is clear that each of these series converges conditionally to $0$.

Before proceeding with a detailed proof of why these $2n$ series have the stated property, we describe the idea behind it; this paragraph can be omitted by readers who just want the detailed proof. Suppose $A \sub \N$ sends the series $\bar a^{n+1}$ to infinity. Because this series only has nonzero terms on blocks of the form $I_{\ell n+1}$, we must have $A \cap I_{\ell n+1} \neq \0$ for infinitely many $\ell$. In fact, we can say more: if $\sum(\bar a^{n+1},A) = \infty$, then $A \cap I_{\ell n+1}$ must contain ``significantly more'' odds than evens, for infinitely many $\ell$. This has an effect on the series $\bar a^1, \bar a^2, \dots, \bar a^n$. Specifically, we must include ``significantly more'' positive than negative terms in the series $\bar a^2, \bar a^3, \dots, \bar a^n$ on the block $I_{\ell n+1}$, and we must include ``significantly more'' negative than positive terms in the series $\bar a^1$ on the block $I_{\ell n+1}$. By making $b_{\ell n+1}$ large enough, we can ensure that including ``significantly more'' positive terms than negative will force the partial sum of a the series $\bar a^2, \bar a^3, \dots, \bar a^n$ to be above $0$ by the end of block $I_{\ell n+1}$, and including ``significantly more'' negative terms than positive will force the partial sum of a the series $\bar a^1$ to be below $0$ at the end of block $I_{\ell n+1}$. Thus having $\sum(\bar a^{n+1},A) = \infty$ forces the partial sums of $\sum(\bar a^{2},A), \dots, \sum(\bar a^{n},A)$ to be positive infinitely often, but it forces the partial sums of $\sum(\bar a^{1},A)$ to be negative infinitely often. Thus, if $\sum(\bar a^{n+1},A) = \infty$, then we cannot have $\sum(\bar a^1,A) = \infty$, and we cannot have $\sum(\bar a^i,A) = -\infty$ for any $2 \leq i \leq n$. Similarly, if $\sum(\bar a^{n+1},A) = -\infty$, then we cannot have $\sum(\bar a^1,A) = -\infty$, and we cannot have $\sum(\bar a^i,A) = \infty$ for any $2 \leq i \leq n$. In other words, if $A$ sends $\bar a^{n+1}$ to infinity, and also sends $\bar a^1$ to infinity and (some of the) $\bar a^i$ as well for $2 \leq i \leq n$, then $\sum(\bar a^1,A)$ must be different from all of the $\sum(\bar a^i,A)$ for $2 \leq i \leq n$; i.e., if the one is $\infty$, then the others are $-\infty$, and vice versa. 
Similarly, sending any $\bar a^{n+j}$ to infinity (where $1 \leq j \leq n$), along with $\bar a^j$ and (some of the) $\bar a^i$ for $1 \leq i \leq n$, $i \neq j$, forces $\sum(\bar a^j,A)$ to be different from all of the $\sum(\bar a^i,A)$ for $1 \leq i \leq n$, $i \neq j$. 
But of course, everyone cannot be ``different'' at the same time. If $\ell$ of the series $\bar a^{n+1}, \bar a^{n+2}, \dots, \bar a^{2n}$ are sent to infinity by some $A \sub \N$, then at most $n+2-\ell$ of the series $\bar a^1, \bar a^2, \dots, \bar a^n$ are sent to infinity by $A$.

We will employ the following notation: given $A \sub \N$, let
\begin{itemize}
\item[$\circ$] $A^+ = \set{i \in \{1,2,\dots,2n\}\vphantom{f^{f^2}}}{\sum(\bar a^i,A) = \infty}$,
\item[$\circ$] $A^- = \set{i \in \{1,2,\dots,2n\}\vphantom{f^{f^2}}}{\sum(\bar a^i,A) = -\infty}$,$\vphantom{f^{f^{f^{f^2}}}}$
\item[$\circ$] $A^\vee = A^+ \cup A^-$.$\vphantom{f^{f^f}}$
\end{itemize}
Our goal is to show $|A^\vee| \leq n+1$ for all $A \sub \N$.

Let $b_1, b_2, \dots, b_m, \dots$ be an increasing sequence of even numbers, with $b_1 = 2$, satisfying the following recurrence relation:
$$b_{m+1} \, \geq \, m^3 (b_1 + b_2 + \dots + b_m)$$
Consider the $2n$ series defined above in terms of the $b_m$, and let $A \sub \N$.

\begin{claim}
Let $i \in \{1,2,\dots,n\}$.
\begin{itemize}
\item[$\circ$] If $n+i \in A^+$, then $i \notin A^+$ and $j \notin A^-$ for all $j \in \{1,2,\dots,n\} \setminus \{i\}$.
\item[$\circ$] If $n+i \in A^-$, then $i \notin A^-$ and $j \notin A^+$ for all $j \in \{1,2,\dots,n\} \setminus \{i\}$.
\end{itemize}
\end{claim}

\begin{proof}[Proof of claim]
Suppose $n+i \in A^+$. For each $m$, let $\Delta(m)$ denote the imbalance of odd terms over even terms in $A$ from block $m$:
\begin{align*}
\Delta(m) \,=\, & \card{\set{\ell \in A \cap I_{m}}{\ell \text{ is odd}}} - \card{\set{\ell \in A \cap I_{m}}{\ell \text{ is even}}}.
\end{align*}
Observe that we may use the quantity $\Delta(m)$ to compute the sum of our $(n+i)^{\mathrm{th}}$ subseries on the $m^{\mathrm{th}}$ block: if $m \equiv j$ (modulo $n$), then
\begin{align*}
\textstyle \sum (\bar a^{n+i},A \cap I_m) =
\begin{cases}
\frac{\Delta(m)}{b_m} & \text{ if } j=i \\
0 & \text{ if } j \neq i.
\end{cases}
\end{align*}
It follows that $\Delta(m) > \nicefrac{b_m}{m^2}$ for infinitely many $m \equiv i$ (modulo $n$) because, if not, then 
$$\textstyle \sum(\bar a^{n+i},A) \,=\, \sum_{m=0}^\infty  \sum(\bar a^{n+i},A \cap I_m) $$
cannot grow fast enough to sum to $\infty$. More precisely, if there were some $M$ such that $\Delta(\ell n+i) \leq \nicefrac{b_{\ell n+i}}{(\ell n+i)^2}$ for every $\ell \geq M$, then
\begin{align*}
\textstyle \sum (\bar a^{n+i},A \cap [Mn,\infty)) & \,=\, \textstyle \sum_{m = Mn}^\infty \sum (\bar a^{n+i},A \cap I_m)  \\
&  \,= \textstyle  \, \sum_{\ell = M}^\infty \frac{\Delta(\ell n+i)}{b_{\ell n+i}} \, \leq \, \sum_{\ell = M}^\infty \frac{1}{(\ell n+i)^2} \, < \, \infty,
\end{align*}
which shows that $\sum (\bar a^{n+i},A)$ converges on a tail, contradicting the assumption that $n+i \in A^+$. Thus, for infinitely many $m \equiv i$ (modulo $n$),
$$\Delta(m) \, > \, \frac{b_m}{m^2} \, \geq \, m(b_1+b_2+\dots+b_{m-1}).$$

Now consider the $i^\mathrm{th}$ subseries $\sum(\bar a^i,A) = \sum_{\ell = 0}^\infty \sum (\bar a^i,A \cap I_\ell)$. 
By the definition of $\bar a^i$, if $m \equiv i$ (modulo $n$), then
$$\textstyle \sum (\bar a^i,A \cap I_m) \, = \, \frac{-\Delta(m)}{m},$$
and in particular, if $\Delta(m) \, > \, m(b_1+b_2+\dots+b_{m-1})$ then
$$\textstyle \sum (\bar a^i,A \cap I_m) \, = \, \frac{-\Delta(m)}{m} \,<\, - b_1-b_2-\dots-b_{m-1}.$$
This negative sum is greater in absolute value than all the preceding terms of the subseries combined. To see this, note that $|a^i_k| = \frac{1}{j}$ whenever $k \in I_{j}$, which implies $\card{\sum (\bar a^i,A \cap I_{j})} \leq |I_{j}|\frac{1}{j} =  \frac{b_{j}}{j}$; thus
\begin{align*}
\textstyle \card{\sum (\bar a^i,A \cap [1,\min I_m))} \, & \leq \, \textstyle \sum_{j < m}\card{\sum (\bar a^i,A \cap I_j)} \\
& \leq \, \textstyle \sum_{j < m} \frac{b_j}{j} \, < \, b_1+b_2+\dots+b_{m-1}.
\end{align*}
Hence, if $m \equiv i$ (modulo $n$) and $\Delta(m) \, > \, m(b_1+b_2+\dots+b_{m-1})$, then
$$\textstyle \sum (\bar a^i,A \cap [0,\max I_m)) \,=\, \sum (\bar a^i,A \cap [0,\min I_m)) + \sum (\bar a^i,A \cap I_m) \,<\, 0.$$
By the previous paragraph, $\Delta(m) \, > \, m(b_1+b_2+\dots+b_{m-1})$ for infinitely many $m \equiv i$ (modulo $n$). Thus the finite partial sums of
$\textstyle \sum (\bar a^i,A)$
are negative infinitely often. It follows that $i \notin A^+$. 

Next consider the $j^\mathrm{th}$ subseries $\sum(\bar a^j,A)$ for some $j \in \{1,2,\dots,n\} \setminus \{i\}$. If $m \equiv i$ (modulo $n$), then $a^j_k = -a^i_k$ for all $k \in I_m$. Thus, if $m \equiv i$ (modulo $n$) and $\Delta(m) \, > \, m(b_1+b_2+\dots+b_{m-1})$ then
$$\textstyle \sum (\bar a^j,A \cap I_m) \,=\, -\sum (\bar a^i,A \cap I_m) \,=\, \frac{\Delta(m)}{m} \,>\, b_1+b_2+\dots+b_{m-1}.$$
Just as in the previous paragraph, this positive sum is greater in absolute value than all the preceding terms of the subseries combined. Hence, as before, if $m \equiv i$ (modulo $n$) and $\Delta(m) \, > \, m(b_1+b_2+\dots+b_{m-1})$ then
$$\textstyle \sum (\bar a^j,A \cap [0,\max I_m)) \,>\, 0.$$
Thus the finite partial sums of
$\textstyle \sum (\bar a^j,A)$
are positive infinitely often, and it follows that $j \notin A^-$. 

An essentially identical argument shows that if $n+i \in A^-$, then $i \notin A^-$ and $j \notin A^+$ for all $j \in \{1,2,\dots,n\} \setminus \{i\}$.
\end{proof}

Let $X = \set{i \in \{1,2,\dots,n\}\vphantom{f^{f^2}}}{i,n+i \in A^\vee}$. The previous claim implies that either $A^+ \cap \{1,2,\dots,n\} = \{i\}$ or else $A^- \cap \{1,2,\dots,n\} = \{i\}$, depending on whether $n+i \in A^-$ or $n+i \in A^+$. It follows that $|X| \leq 2$, and from this is follows that  $|A^\vee| \leq n+2$.
\end{proof}

\section{The infinite version}\label{sec:infinite}

In this section we prove an infinite version of Theorem~\ref{thm:main}:

\begin{theorem}\label{thm:infinite}
Let $\set{\bar a^i}{i \in \w}$ be a countably infinite collection of conditionally convergent series. There is some set $A \sub \N$ such that $\sum(\bar a^i,A) = \infty$ for infinitely many $i$.
\end{theorem}

As in the previous section, we will write $\bar a^i$ to denote the infinite sequence $\seq{a^i_n}{n \in \N}$.

\begin{definition}
Suppose $\C = \set{\bar a^i}{i \in \w}$ is a collection of conditionally convergent series. $A \sub \N$ is called \emph{tame} with respect to $\C$ if for each $i \in \w$, all the terms of the subseries $\sum(\bar a^i,A)$ have the same sign, with at most finitely many exceptions (not counting zeros). If the collection $\C$ is clear from context, we simply say that $A$ is tame.
\end{definition}

Note that the tameness of $A$ can be expressed in terms of incompatible pairs of ideals: $A \sub \N$ is tame if for every $i \in \w$, $A$ is a member of at least one of the incompatible ideals $\I = \set{X \sub \N}{\set{n \in X}{a^i_n > 0} \text{ is finite}}$ and $\J = \set{X \sub \N}{\set{n \in X}{a^i_n < 0} \text{ is finite}}$.

\begin{notation}
Given a collection $\C = \set{\bar a^i}{i \in \w}$ of conditionally convergent series and $A \sub \N$, 
we define (as in the proof of Theorem~\ref{thm:lbser})
\begin{itemize}
\item[$\circ$] $A^+ = \set{i \in I}{\textstyle \sum(\bar a^i,A) = \infty}$, 
\item[$\circ$] $A^- = \set{i \in I}{\textstyle \sum(\bar a^i,A) = -\infty}$, and
\item[$\circ$] $A^\vee = A^+ \cup A^-$.
\end{itemize}
\end{notation}

\begin{definition}
Suppose $\C = \set{\bar a^i}{i \in \w}$ is a collection of conditionally convergent series. If $f$ is a function from a subset of $\w$ to $\{1,-1\}$, we say that $f$ \emph{is represented by} $A \sub \N$ if $f^{-1}(1) = A^+$ and $f^{-1}(-1) = A^-$.
For each $A \sub \N$, let $f_A$ denote the function represented by $A$; that is, define
$$f_A(i) \,= \, 
\begin{cases}
1 & \text{ if } \sum(\bar a^i,A) = \infty, \\
-1 & \text{ if } \sum(\bar a^i,A) = -\infty
\end{cases}$$
and leave $f_A(i)$ undefined otherwise.
\end{definition}

\begin{lemma}\label{lem:up/down}
Let $\bar a$ be a conditionally convergent series, and let $A \subseteq \N$. If $\sum(\bar a,A) = -\infty$ then $\sum(\bar a,\N \setminus A) = \infty$.
\end{lemma}

\begin{lemma}\label{lem:tame}
Let $\set{\bar a^i}{i \in \w}$ be a countable collection of conditionally convergent series. For each $i \leq \w$, there is a tame $A \sub \N$ with $i \in A^+$.
\end{lemma}
\begin{proof}
For convenience, let us set $i = 0$ and show that there is some tame $A \sub \N$ with $\sum(\bar a^0,A) = \infty$. 
An essentially identical argument works for any other value of $i$.

We begin with a recursive construction of a decreasing sequence $C_0 \supseteq C_1 \supseteq C_2 \supseteq \dots$ of subsets of $\N$ and an increasing sequence $k_0 < k_1 < k_2 < \dots$ of non-negative integers. In the end, we will take
$A = \textstyle \bigcup_{n \in \w}C_n \cap (k_n,k_{n+1}].$

To begin, take $C_0 = \set{n \in \N}{a^0_n \geq 0}$. Let $k_0 = 0$ and let $k_1$ be the smallest natural number with the property that
$$\textstyle \sum(\bar a^0,C_0 \cap [1,k_1]) \geq 1.$$
Some such number $k_1$ must exist by our choice of $C_0$.

For the recursive step, we begin with sets $C_0 \supseteq C_1 \supseteq \dots \supseteq C_{\ell-1}$, and with natural numbers $0 = k_0 < k_1 < k_2 < \dots < k_\ell$ satisfying the following inductive hypotheses:
\begin{itemize}
\item[$\circ$] $C_0 \supseteq C_1 \supseteq \dots \supseteq C_{\ell-1}$ and $k_0 < k_1 < k_2 < \dots < k_\ell$.
\item[$\circ$] For each $i \leq j < \ell$, all the numbers in $\set{a^i_n}{n \in C_j}$ have the same sign. 
\item[$\circ$] $\sum(\bar a^0,C_{\ell-1}) = \infty$.
\item[$\circ$] For each $i < \ell$, $\sum(\bar a^0,C_i \cap (k_{i},k_{i+1}]) \geq 1$.
\end{itemize}
We then partition $C_{\ell-1}$ into two sets as follows:
$$C_{\ell-1}^{\geq} = \set{n \in C_{\ell-1}}{a^\ell_n \geq 0},$$
$$C_{\ell-1}^{<} = \set{n \in C_{\ell-1}}{a^\ell_n < 0}.$$
Because $\sum(\bar a^0,C_{\ell-1}) = \infty$, and because every term in this sum is positive, we must have either $\sum(\bar a^0,C_{\ell-1}^\geq) = \infty$ or $\sum(\bar a^0,C_{\ell-1}^<) = \infty$ (possibly both). We choose $C_\ell$ to be either of $C_{\ell-1}^\geq$ or $C_{\ell-1}^<$, so long as $\sum(\bar a^0,C_\ell) = \infty$. Then let $k_{\ell+1}$ be the smallest natural number with the property that 
$$\textstyle \sum(\bar a^0,C_\ell \cap (k_\ell,k_{\ell+1}]) \geq 1.$$
Some such number $k_{\ell+1}$ must exist by our choice of $C_\ell$.
This completes the recursive step of the construction, and it is not difficult to see that all four of the above hypotheses remain true for the next stage of the recursion.

The result of this construction is a decreasing sequence $C_0 \supseteq C_1 \supseteq C_2 \supseteq C_3 \supseteq \dots$ of subsets of $\N$ and an increasing sequence $k_0 < k_1 < k_2 < k_3 < \dots$ of non-negative integers such that
\begin{itemize}
\item[$\circ$] for each $i \leq j \in \w$, all the numbers in $\set{a^i_n}{n \in C_j}$ have the same sign, and
\item[$\circ$] for each $i \in \w$, $\sum(\bar a^0,C_i \cap (k_{i},k_{i+1}]) \geq 1$.
\end{itemize}
Let $A = \textstyle \bigcup_{n \in \w}C_n \cap (k_n,k_{n+1}].$ For each $i \in \w$, the series $\sum(\bar a^i,A)$ consists of terms all having the same sign, with a finite number of exceptions (specifically, the possible exceptions are only the terms with index $< k_i$). Hence $A$ is tame. 
Also, $\sum(\bar a^0,A) = \infty$ because each term of this sum is positive, and for each $i \in \N$ the finite partial sum $\sum(\bar a,A \cap [1,k_{i+1}])$ is at least $\sum_{j = 0}^i\sum(\bar a,A \cap (k_j,k_{j+1}]) \,=\, \sum_{j = 0}^i\sum(\bar a,C_j \cap (k_j,k_{j+1}]) \,\geq\, i+1$.
\end{proof}

\begin{proof}[Proof of Theorem~\ref{thm:infinite}]
Let $\set{\bar a^i}{i \in \w}$ be a countably infinite collection of conditionally convergent series. 
Suppose for some $A \sub \N$ that the domain of $f_A$ is infinite.
If $f_A^{-1}(1)$ is infinite, then $\sum(\bar a^i,A) = \infty$ for infinitely many $i$, and we are done. If not, then $f_A^{-1}(-1)$ must be infinite, in which case $\sum(\bar a^i,\N \setminus A) = \infty$ for infinitely many $i$ by Lemma~\ref{lem:up/down}, and again we are done.

Thus, in order to prove the theorem, it suffices to show that for some $A \sub \N$, the function $f_A$ has infinite domain. Aiming for a contradiction, let us suppose the opposite: that every $A \sub \N$ has finite domain.

Let $\F = \set{f_A}{A \sub \N \text{ is tame}}$. Let us say that a subset $\G$ of $\F$ is \emph{large} if for infinitely many $i \in \w$, there exists some $f \in \G$ such that $f(i) = 1$. 
Observe that $\F$ is large by Lemma~\ref{lem:tame}, and that if a large subset of $\F$ is partitioned into finitely many pieces, then one of those pieces must also be large.

We now use recursion to define a function $g: \w \to \{-1,0,1\}$. This function is constructed via a recursively defined sequence of increasingly large approximations to $g$, namely finite partial functions $g_0, g_1, g_2, \dots$. Along with the $g_i$, we also construct a decreasing sequence $\F_0 \supseteq \F_1 \supseteq \F_2 \supseteq \dots$ of large subsets of $\F$.

To begin, let $\F_0 = \F$ and $g_0 = \0$. At stage $\ell$ of the recursion, we begin with a finite partial function $g_{\ell}: \{0,1,2,\dots,\ell-1\} \to \{-1,0,1\}$ and a large $\F_{\ell} \sub \F$. Partition $\F_{\ell}$ into three parts as follows:
$$\F_{\ell}^1 = \set{f \in \F_{\ell}}{f(\ell) = 1},$$
$$\F_{\ell}^{-1} = \set{f \in \F_{\ell}}{f(\ell) = -1},$$
$$\F_{\ell}^0 = \set{f \in \F_{\ell}}{\ell \notin \mathrm{dom}(f)}.$$
Because $\F_\ell$ is large, at least one of these three pieces is large as well; choose some such piece $\F^j_{\ell}$. Then set $\F_{\ell+1} = \F^j_\ell$ and set $g_\ell = g_{\ell-1} \cup \{(\ell,j)\}$.

The result of this construction is a function $g: \w \to \{-1,0,1\}$ with the property that, for every $\ell \in \w$, the set of all $f \in \F$ satisfying
\begin{itemize}
\item[$(*)$] For $i < \ell$, $i \in \mathrm{dom}(f)$ if and only if $g(i) \neq 0$, and furthermore $f(i) = g(i)$ for $i \in \mathrm{dom}(f)$.
\end{itemize}
is large, because it contains $\F_\ell$.

We now use $g$ as a ``guide'' for recursively constructing some $A \sub \N$ with $\mathrm{dom}(f_A)$ infinite. There are three cases to consider.

\vspace{3mm}
\noindent \emph{Case 1: } Suppose $g(i) = 1$ for infinitely many $i \in \w$.
\vspace{2mm}

In this case, we recursively construct a sequence $A_0, A_1, A_2 \dots$ of tame subsets of $\N$. These sets can then be combined to form the desired set $A$.

To begin the construction, let $i_0$ be the first natural number for which $g(i_0) = 1$, and choose some tame $A_0 \sub \N$ such that $f_{A_0}(i_0) = 1$. Some such $A_0$ exists by property $(*)$ above, and because $g(i_0) = 1$. 

For the recursive step, we begin with natural numbers $i_0 < i_1 < \dots < i_{k-1}$ such that $g(i_0) = g(i_1) = \,\dots\, = g(i_{k-1}) = 1$, and with $A_0, A_1, \dots, A_{k-1} \sub \N$. Let $i_k$ denote the first natural number such that
$$\textstyle i_k \,>\, \max \left( \bigcup_{j < k}\mathrm{dom}(f_{A_j}) \right)$$
(recalling that, by assumption, $f_A$ has finite domain for every $A \sub \N$) and such that 
$g(i_k) = 1.$
Then choose some tame $A_k \sub \N$ such that 
$$f_{A_k}(i_0) = f_{A_k}(i_1) = \,\dots\, = f_{A_k}(i_k) = 1.$$
Some such $A_k$ exists because $g(i_0) = g(i_1) = g(i_2) = \,\dots\, = g(i_k) = 1$. In fact, our choice of $g$ ensures that a ``large'' set of $f_{A_k}$ have this property.

The result of this construction is a sequence $A_0, A_1, A_2, \dots$ of tame subsets of $\N$ with the property that, for all $k$, $f_{A_k}(i_0) = f_{A_k}(i_1) = \,\dots\, = f_{A_k}(i_k) = 1$, and $i_\ell \notin \mathrm{dom}(f_{A_k})$ for all $\ell > k$.
In particular, we have
\begin{itemize}
\item[$\circ$] for each $\ell \leq k$, $\sum(\bar a^{i_\ell},A_k) = \infty$, and all but finitely many terms of this sum have the same sign (necessarily positive), and
\item[$\circ$] for each $\ell > k$, $\sum(\bar a^{i_\ell},A_k)$ converges absolutely. (This follows from the tameness of $A_k$ together with the fact that $i_\ell \notin \mathrm{dom}(f_{A_k})$.)
\end{itemize}

The first bullet point implies that for each $k \in \w$, there is some $M_k \in \N$ such that $a_n^{i_\ell} \geq 0$ for all $\ell \leq k$ and $n \geq M_k$. Let
$$\textstyle A = \bigcup_{k \in \w} A_k \cap [M_k,\infty).$$
We claim that $\sum(\bar a^{i_k},A) = \infty$ for every $i_k$, $k \in \w$. To see this, fix $k \in \w$ and partition $A$ into the following two sets:
\begin{align*}
A_{<k} & =\, \textstyle \bigcup_{\ell < k} A_\ell \cap [M_\ell,\infty) \\
A_{\geq k} & =\, \textstyle \bigcup_{\ell \geq k} A_\ell \cap [M_\ell,\infty).
\end{align*}
It is clear that $\sum(\bar a^{i_k},A_{< k})$ converges absolutely and that $\sum(\bar a^{i_k},A_{\geq k}) = \infty$. Because $A = A_{<k} \cup A_{\geq k}$, it follows that $\sum(\bar a^{i_k},A) = \infty$, as desired.

\vspace{3mm}
\noindent \emph{Case 2: } Suppose $g(i) = -1$ for infinitely many $i \in \w$.
\vspace{2mm}

Proceeding just as in Case 1, we may find a set $A \sub \N$ such that $\sum(\bar a^i,A) = -\infty$ for infinitely many $i \in \w$.

\vspace{3mm}
\noindent \emph{Case 3: } Suppose $g(i) = 0$ for infinitely many $i \in \w$.
\vspace{2mm}

As in case 1, we construct a sequence $A_0, A_1, A_2 \dots$ of tame subsets of $\N$ along with a sequence $i_0 < i_1 < i_2 < \dots$ of natural numbers, and afterward these will be used to define $A$.

To begin, let $i_0$ be the least natural number such that $g(i_0) = 0$, and choose some tame $A_0 \sub \N$ such that $f_{A_0}(i_0) = 1$. Some such $A_0$ exists by Lemma~\ref{lem:tame}. 

For the recursive step, we begin with natural numbers $i_0 < i_1 < \dots < i_{k-1}$ and with $A_0, A_1, \dots, A_{k-1} \sub \N$. Recall that, by the definition of $g$, there is a large set $\G$ of functions $f_A$ such that $f_A(i_0) = f_A(i_1) = \,\dots\, = f_A(i_{k-1}) = 0$. By the definition of ``large'' we may find some tame $A_k \sub \N$ such that 
$$f_{A_k}(i_0) = f_{A_k}(i_1) = \,\dots\, = f_{A_k}(i_{k-1}) = 0$$
and such that $f_{A_k}(i_k) = 1$ for some $i_k > \max \left( \bigcup_{j < k}\mathrm{dom}(f_{A_j}) \right)$.

The result of this construction is a sequence $A_0, A_1, A_2, \dots$ of tame subsets of $\N$ with the property that, for all $k$, $f_{A_k}(i_k) = 1$, and $i_\ell \notin \mathrm{dom}(f_{A_\ell})$ for all $\ell \neq k$. In particular, we have
\begin{itemize}
\item[$\circ$] $\sum(\bar a^{i_k},A_k) = \infty$ for each $k$, and
\item[$\circ$] $\sum(\bar a^{i_\ell},A_k)$ converges absolutely for each $\ell \neq k$.
\end{itemize}

For each $k$, choose $M_k \in \N$ large enough so that
$$\textstyle \sum\left(\seq{|a^{i_\ell}_n|}{n \in \N},A_k \cap [M_k,\infty)\right) < \nicefrac{1}{2^k}$$
for each $\ell < k$, which is possible because each of the finitely many series $\sum(\bar a^{i_\ell},A_k)$, for $\ell < k$, converges absolutely. Let
$$\textstyle A = \bigcup_{k \in \w} A_k \cap [M_k,\infty).$$
To finish the proof, we claim that $\sum(\bar a^{i_k},A) = \infty$ for every $i_k$, $k \in \w$. To see this, fix $k \in \w$ and partition $A$ into the following three sets:
\begin{align*}
A_{<k} & =\, \textstyle \bigcup_{\ell < k} A_\ell \cap [M_\ell,\infty) \\
A_{k} \ \, & =\, A_k \cap [M_k,\infty) \\
A_{>k} & =\, \textstyle \bigcup_{\ell > k} A_\ell \cap [M_\ell,\infty).
\end{align*}
The series $\sum(\bar a^{i_k},A_{>k})$ converges absolutely, because
$$\textstyle \sum\left(\seq{|a^{i_\ell}_n|}{n \in \N},A_{>k}\right) \qquad \qquad \qquad \qquad \qquad \qquad$$
$$\qquad \qquad \leq\, \sum_{\ell > k} \textstyle \sum \left( \seq{|a^{i_\ell}_n|}{n \in \N},A_\ell \cap [M_\ell,\infty) \right)
\displaystyle \leq\, \sum_{\ell > k} \nicefrac{1}{2^\ell} \,=\, \frac{1}{2^{k-1}}.$$
The series $\sum(\bar a^{i_k},A_{<k})$ converges absolutely, as $A_{<k} = \bigcup_{\ell < k} A_\ell \cap [M_\ell,\infty)$ and each of the (finitely many) series $\sum(\bar a^{i_k},A_\ell \cap [M_\ell,\infty))$, where $\ell < k$, converges absolutely.
On the other hand, $\sum(\bar a^{i_k},A_k) = \infty$. Because $A = A_{<k} \cup A_k \cup A_{>k}$, it follows that $\sum(\bar a^{i_k},A) = \infty$ as desired.
\end{proof}

This theorem raises the question of whether there is a corresponding infinite version of the main theorem, dealing with arbitrary pairs of incompatible ideals:
\begin{itemize}
\item[$\circ$] Given an infinite sequence $(\I_1,\J_1), (\I_2,\J_2), (\I_2,\J_2),\dots$ of incompatible ideals on some set $X$, is there a single $A \sub X$ that chooses between infinitely many of these pairs?
\end{itemize}
The answer to this question is negative, as the following example shows.

\begin{example}
Let $X = 2^\w$ denote the set of all functions $\w \to \{0,1\}$, and let $X$ be endowed with its usual product topology as the Cantor space. 
For each $n \in \w$, let
$\I_n$ denote the set of all $A \sub X$ that are nowhere dense in $\set{x \in X}{x(n) = 0}$, and let $\J_n$ denote the set of all $A \sub X$ that are nowhere dense in $\set{x \in X}{x(n) = 1}$.

It is not difficult to check that $\I_n$ and $\J_n$ are incompatible ideals on $X$ for all $n \in \w$. We claim that any given $A \sub X$ cannot choose between infinitely many of the pairs $(\I_n,\J_n)$. Indeed, if $A$ chooses between $\I_n$ and $\J_n$ for any $n$, then $A$ is somewhere dense in $X$, i.e., $\closure{A}$ contains a basic open subset of $X$. In other words, $\closure{A}$ contains a set of the form
$$[s] \,=\, \set{x \in X}{x \rest \mathrm{length}(s) = s},$$
where $s$ is a function $m \to \{0,1\}$ for some $m \in \w$. 
But this implies $A \notin \I_k$ and $A \notin \J_k$ for all $k \geq m$. Thus $A$ does not choose between $\I_k$ and $\J_k$ for any $k \geq m$. $\hfill \tiny{\qed}$
\end{example}

Note that this example can be modified to make the set $X$ countable: simply replace $2^\w$ with a countable dense subset of $2^\w$ in the example above. We can also modify the example to make the ideals $\I_n$ and $\J_n$ into $P$-ideals: simply replace ``nowhere dense'' with ``meager'' in the example above, and note that if $A \sub 2^\w$ is non-meager, then there is some open $U$ such that $A \cap V$ is non-meager for every open $V \sub U$.

In closing, let us point out that the ideas in this paper emerged from set-theoretic investigations into cardinal characteristics of the continuum in \cite{BBH}. In the course of these investigations, the question arose: \emph{How small can a collection $\C$ of conditionally convergent series be with the property that every $A \sub \N$ fails to send some member of $\C$ to infinity?} (Specifically, the answer to this question is the so-called ``Galois-Tukey dual'' -- see \cite{Vojtas}, or \cite[Section 4]{Blass} -- of the uncountable cardinal $\ss_i$ as defined in \cite{BBH}.) We suspected that the answer to this question should be an uncountable cardinal number. It was a surprise to discover that the correct answer is $4$, or, in the terminology of Section~\ref{sec:series}, the least $n$ such that $n \neq \ser(n)$. This surprise led us to investigate the function $\ser(n)$ generally, and its upper bounds $\Ser(n)$ and $\Hyp(n)$.

In the same way that the function $\ser(n)$ is related to the cardinal invariant $\ss_i$ from \cite{BBH}, the function $\Ser(n)$ suggests a new cardinal invariant of the continuum, a natural upper bound for $\ss_i$ . Namely, this cardinal invariant is defined as the answer to the following question: \emph{How small can a collection $\A$ of subsets of $\N$ be with the property that for every pair of incompatible ideals on $\N$, some $A \in \A$ chooses between them?} To end our paper, we show that this ``new'' cardinal invariant is not really new at all:

\begin{theorem}
Suppose $\A$ is a collection of subsets of $\N$ such that for every pair of incompatible ideals on $\N$, some $A \in \A$ chooses between them. Then $\card{\A} = \continuum$.
\end{theorem}
\begin{proof}
Let $\A$ be any collection of subsets of $\N$ such that for every pair of incompatible ideals on $\N$, some $A \in \A$ chooses between them. Let $\C$ denote the closure of $\A$ under the Boolean operations of taking complements, finite unions, and finite intersections. Because $|\C| = \max \{\aleph_0,|\A|\}$, to prove the theorem it suffices to show that $|\C| = \continuum$.

Our proof uses the topological space $\N^*$, the space of all non-principal ultrafilters on $\N$. Recall that if $A \sub \N$, then $A^* = \set{u \in \N^*}{A \in u}$ is a clopen subset of $\N^*$, and the sets of this form constitute a basis for $\N^*$.

For each $u \in \N^*$, let $\hat u = \set{A \sub \N}{\N \setminus A \in u} = \set{A \sub \N}{A \notin u}$. Each $\hat u$ is an ideal on $\N$. Furthermore, if $u,v \in \N^*$ and $u \neq v$, then there is some $A \sub \N$ such that $\N \setminus A \in \hat u$ and $A \in \hat v$. Thus for any distinct $u,v \in \N^*$, $\hat u$ and $\hat v$ are incompatible ideals on $\N$. 

Therefore, whenever $u,v \in \N^*$ with $u \neq v$, some $A \in \C$ chooses between $\hat u$ and $\hat v$. This could mean that $A \in \hat u \setminus \hat v$, in which case $A \in \hat v$ and $\N \setminus A \in \hat u$, or else that $A \in \hat v \setminus \hat u$, in which case $A \in \hat u$ and $\N \setminus A \in \hat v$. In either case (because $\C$ is closed under taking complements), we see that there are $A,A' \in \C$ such that $u \in A^*$ and $v \in (A')^*$.

Let $\C^* = \set{A^*}{A \in \C}$. By the previous paragraph, $\C^*$ separates points in $\N^*$, meaning that for any two distinct points of $\N^*$, there are disjoint sets in $\C^*$ each containing one of the two points. We claim that this implies $\C^*$ is a basis for $\N^*$. To see this, first note that, because $\C$ is closed under finite unions and intersections, $\C^*$ is as well (because $(A \cup B)^* = A^* \cup B^*$ and $(A \cap B)^* = A^* \cap B^*$ for any $A,B \sub \N$). Thus $\C^*$ forms the basis for some topology $\tau$ on $\N^*$. The topology $\tau$ is Hausdorff (because $\C^*$ separates points), and is coarser than the usual topology on $\N^*$. But the usual topology on $\N^*$ is compact and Hausdorff, and it is well-known that no Hausdorff topology is strictly coarser than a compact Hausdorff topology \cite[Corollary 3.1.14]{Engelking}. Thus $\tau$ is coarser than the usual topology on $\N^*$, but not strictly coarser: that is, they are the same. Hence $\C^*$ is a basis for the usual topology on $\N^*$.

The usual topology on $\N^*$ has no basis of size $<\!\continuum$ \cite[Theorem 3.6.14]{Engelking}, so the previous paragraph implies $|\C| = \continuum$ as claimed.
\end{proof}



\end{document}